\documentclass[3p]{elsarticle}

\usepackage{hyperref,amsmath,amsfonts,amssymb,amsthm}
\usepackage[mathcal]{euscript}
\usepackage{xcolor}
\hypersetup{
	colorlinks=true,
	linkcolor=teal,
	citecolor=magenta,
	urlcolor=teal
}

\journal{Journal of Functional Analysis}

\let\embed=\hookrightarrow
\def\R{\mathbb{R}}
\def\N{\mathbb{N}}
\def\d{{\fam0 d}}
\def\df#1{{\it #1}\/}
\def\upto{\uparrow}
\let\embed=\hookrightarrow
\def\eprec{\prec\kern-3pt\prec}
\def\esucc{\succ\kern-3pt\succ}
\DeclareMathOperator*{\esssup}{ess\,sup}
\def\Tr{\mathop{{\fam0 Tr}}\nolimits}
\def\dist{\mathop{{\fam0 dist}}\nolimits}
\newtheorem{theorem}{Theorem}[section]
\newtheorem*{theoremA}{Theorem A}
\newtheorem*{theoremB}{Theorem B}
\newtheorem*{theoremC}{Theorem C}
\newtheorem{lemma}[theorem]{Lemma}
\newtheorem{proposition}[theorem]{Proposition}
\newdefinition{remark}[theorem]{Remark}
\newdefinition{example}[theorem]{Example}
\newproof{pfoof}{Proof}
\numberwithin{equation}{section}
\let\bar\overline

\begin{document}

\begin{frontmatter}

\title{Optimal Orlicz domains in Sobolev embeddings into Marcinkiewicz spaces}

\author{V\'\i t Musil}
\address{Department of Mathematical Analysis, Faculty of Mathematics and Physics, Charles
University, Sokolovsk\'a 83, 186 75 Praha 8, Czech Republic}
\ead{musil@karlin.mff.cuni.cz}

\begin{abstract}
In this paper we present a method for determining whether there exists a largest
Orlicz space $L^A(\Omega)$ satisfying the Sobolev embedding
$$W^mL^A(\Omega) \embed Y(\Omega)$$
where $Y(\Omega)$ stands for an arbitrary so-called Marcinkiewicz endpoint space.
The tool developed in this work enables us to investigate the optimality of Orlicz
domain spaces in Sobolev embeddings and also in Sobolev trace embeddings on
domains $\Omega$ in $\R^n$ with various regularity.
\end{abstract}

\begin{keyword}
Optimal Orlicz domain \sep
Sobolev embedding \sep
Orlicz space \sep
Marcinkiewicz space
\MSC[2010]
46E30 \sep
46E35 \sep
47G10 \sep
26D10 \sep
\end{keyword}

\end{frontmatter}


\section*{How to cite this paper}
\noindent
V. Musil, Optimal Orlicz domains in Sobolev embeddings into Marcinkiewicz
spaces, \emph{J. Funct. Anal}. 270 (2016), no. 7, 2653-2690.
\begin{center}
	\url{http://doi.org/10.1016/j.jfa.2016.01.019}
\end{center}

\section{Introduction and main results}

\hyphenation{Mar-cin-kie-wicz}
\def\domain{0,1}
\def\ast{^*}

For a given Banach function space $Y(\Omega)$, we study the question whether there exists
an optimal (i.e.\ largest) Orlicz space $L^A(\Omega)$
satisfying the embedding
$$
	W^mL^A(\Omega) \embed Y(\Omega),
$$
where $\Omega$ stands for a bounded domain in $\R^n$, $n\ge 2$, and
$W^mL^A(\Omega)$ is an Orlicz-Sobolev space
(for the definition see the section below).
By optimality we mean that the
space $L^A(\Omega)$ cannot be replaced by any strictly bigger Orlicz space, i.e.,
every embedding of an Orlicz-Sobolev space to~$Y(\Omega)$ factorizes through
the space~$W^mL^A(\Omega)$.

In the general setting of rearrangement-invariant (r.i.)\ Banach function
spaces, such questions were investigated using the method of reducing the
Sobolev embeddings to the boundedness of an appropriate modification of the
weighted Hardy operator.  In the setting of r.i.\ spaces, the optimal domain
and the optimal target spaces are then explicitly described (see~\cite{CPS},
\cite{EKP}, \cite{KP}, \cite{KP2}).

However, for certain specific applications such as to the solution of partial
differential equations, it is often useful to investigate the optimality of
spaces in Sobolev-type embeddings restricted to the context of Orlicz spaces.
This creates a difficult and important problem that has been studied by several
authors (see e.g.~\cite{AC1}, \cite{AC2}, \cite{AC3}, \cite{AC}, \cite{CP3}, \cite{CPS},
\cite{G}, \cite{HMT}). In particular, the
situation in this setting is significantly different than in the broader sense of
r.i.\ spaces.

Consider, for instance, the well-known classical Sobolev embedding
$W^1L^p(\Omega)\embed L^{p\ast}(\Omega)$, where $1<p<n$, $p^*=np/(n-p)$ and
$\Omega$ has a Lipschitz boundary. Both the spaces $L^p(\Omega)$ and
$L^{p\ast}(\Omega)$ that appear in this embedding are clearly optimal in the
context of Lebesgue spaces, the former as the domain and the latter as the
range. It turns out that they are optimal even in the broader context of Orlicz
spaces, but that is a deeper observation and more difficult to prove. The
optimality of the range space $L^{p\ast}(\Omega)$ follows from a general
result of A. Cianchi \cite{AC}. On the other hand, the optimality of the
domain space $L^p(\Omega)$ has not been known so far and will follow from
our more general statement below (Example~\ref{ex:extension_to_Y}).

In the limiting case when $p=n$, the situation is different and more
interesting. First, if we fix the domain space $L^n(\Omega)$, then there is
no optimal range Lebesgue space $L^q(\Omega)$ that would render the
embedding $W^{1}L^n(\Omega)\embed L^{q}(\Omega)$ true, because it holds
for every $q<\infty$, but not for $q=\infty$. This discrepancy
was remedied in the 1960s by a clever use of special Orlicz spaces of an
exponential type. In particular, by now classic results of
N.\,S.\,Trudinger, S.\,I.\,Pokhozhaev, and V.\,I.\,Yudovich (see \cite{Po},
\cite{T}, \cite{Yu}), one has
$$
W^{1}L^n(\Omega)\embed \exp L^{n'}(\Omega),
$$
where $n'=n/(n-1)$. Now, both the domain space $L^{n}(\Omega)$ and the range
space $\exp L^{n'}(\Omega)$ are Orlicz spaces, and therefore we may ask,
again, about their optimality. It turns out that, while the target space is the
optimal (that means smallest) Orlicz space that renders this Sobolev embedding
true (this was originally proved by J.\,A.\,Hempel, G.\,R.\,Morris and
N.\,S.\,Trudinger \cite{HMT} and it also follows from a general result of A.\
Cianchi \cite{AC}), the domain space is not. Rather surprisingly, it can even be
shown that such an optimal Orlicz domain space does not exist at all. More
precisely, given an Orlicz space $L^{A}(\Omega)$ such that
$W^{1}L^A(\Omega)\embed \exp L^{n'}(\Omega)$, there always exists another
Orlicz space $L^{B}(\Omega)$, strictly bigger than $L^{A}(\Omega)$ such
that $W^{1}L^B(\Omega)\embed \exp L^{n'}(\Omega)$. This result was shown in
\cite{Bonn}.

It is clear from these examples that even the very existence of an optimal
Orlicz partner (either range or domain) is highly nontrivial and very
interesting. However, the question of existence (and, possibly,
characterization) of an optimal Orlicz domain partner, is of interest also in a
more general situation when the given target space is not necessarily an Orlicz
space. 

For instance, one has the embedding
$$
	W^1L^p(\Omega)\embed L^{p\ast,\,p}(\Omega),
$$
(see e.g.~\cite{Hu}, \cite{Ma}, \cite{ON}, \cite{Pe}) in which the target
space is a usual two-parameter Lorentz space. Moreover, it is known that $
L^{p\ast,\,p}(\Omega)$ is  the optimal r.i.\ range space in this embedding and the
space $L^p(\Omega)$ is the optimal r.i.\ domain space (see~\cite{EKP}
or~\cite{CPS}). Therefore, $L^p(\Omega)$ is automatically also the optimal
Orlicz space in this embedding.

On the other hand, when we start with the space $L^\infty(\Omega)$ at the
position of the range space, then, again, as A.\,Cianchi and L.\,Pick showed in
\cite{CP}, an optimal Orlicz space does not exist at all. This situation
resembles the above-mentioned embedding in which the target was the space $\exp
L^{n'}(\Omega)$. Apart from these two very particular cases the question of the existence
of an optimal Orlicz space has been open.

The general question of optimality among the Orlicz spaces has already been
studied (see \cite{AC3}, \cite{AC}, \cite{CP2}, \cite{CP3}, \cite{CPS},
\cite{G}) however, all those papers focus on the optimality of target spaces. In
the case of range, it turns out that the answer is always  affirmative, and,
furthermore, an explicit description of the optimal Orlicz space is available.
The situation is however dramatically different when the target space is fixed
and the optimality of the domain space is in question.

In this paper we study this question in the special case when the target space
is chosen from the class of the so-called Marcinkiewicz endpoint spaces. This
is not as restrictive as it may seem since the most customary cases
including those given by the previous examples are covered. 

An important ingredient of our approach is the use of known reduction theorems
(see \citep[Theorems 6.1 and 6.4]{CPS} and \citep[Theorem 1.3]{CP2}).
This method will enable us to circumvent working with Sobolev spaces to
consider instead the boundedness of the operator
$$
	\bigr(H_{\alpha}^{\beta} f\bigr)(t) := \int_{t^\beta}^1 f(s)\,s^{\alpha - 1}\,\d s,\quad t\in(0,1),
$$
in one dimension. Here $0<\alpha<1$, $0<\beta<\infty$ and $\alpha+{1/\beta}
\ge 1$.  Then, by using various special cases of $\alpha$ and $\beta$ we obtain
applications not only to Sobolev embeddings but also to the trace Sobolev
embeddings of different orders and on various domains in $\R^n$ at once.

Now we are in a position to state our main result which gives a complete
characterization of when the optimal Orlicz domain exists, and also its
explicit description.  Simply put, to a given Marcinkiewicz endpoint space
$M(\Omega)$ we construct an ``optimal Orlicz candidate'' $L^B(\Omega)$ in terms
of the fundamental function. We exploit the fact that to a given fundamental
function there always exists a uniquely defined Orlicz space. Next, we test
whether the embedding $W^mL^B(\Omega)\embed M(\Omega)$ holds. If so, then we
show that $L^B(\Omega)$ is the optimal Orlicz domain. Otherwise, we can prove
that an optimal Orlicz domain does not exist at all. The general result reads
as follows.

	\begin{theoremA} \label{thm:characterization}
		Let $0<\alpha<1$, $\beta>0$, $\alpha+1/\beta\ge 1$
		and let $M(\domain)$ be a Marcinkiewicz endpoint space
		with a fundamental function $\varphi$ satisfying
		$$
			\sup_{0<t<1} \varphi(t^{1\over\beta})\, t^{{\alpha}-1}=\infty.
		$$
		Let $X(\domain)$ be the largest r.i.\ space satisfying
		$$
			H_\alpha^\beta\colon X(\domain) \to M(\domain).
		$$
		Denote by $L^B(\domain)$ the Orlicz space having the same fundamental
		function as the space $X(\domain)$.
		Then the following statements are equivalent.
		\begin{enumerate}[(i)]
			\item There exists a largest Orlicz space $L^A(0,1)$ satisfying the relation
			$$
				H_\alpha^\beta\colon L^A(\domain) \to M(\domain);
			$$
			\item
			$$
				H_\alpha^\beta\colon L^B(\domain) \to M(\domain);
			$$
			\item
			$$
				L^B(\domain) \subseteq X(\domain);
			$$
			\item
			$$
				S_\alpha\colon L^{\widetilde{B}}(0,1) \to L^{\widetilde{B}}(0,1),
			$$
			where $S_\alpha$ is the operator given by
			$$
				\bigl(S_\alpha f\bigr)(t) := t^{\alpha - 1} \sup_{0<s<t} s^{1-\alpha}\,f^*(s),
					\quad t\in(0,1);
			$$

			\item there exists some $K\ge 1$ such that
			$$
				\int_1^t {\widetilde{B}(s)\over s^{{1/(1-\alpha)}+1}}\,\d s
					\lesssim {\widetilde{B}(Kt) \over t^{{1/(1-\alpha)}}},
						\quad t\in (2,\infty).
			$$

			\noindent Moreover, if $\widetilde{B}$ satisfies the $\Delta_2$ condition, then each of the conditions
			{\rm (i)--(v)} is equivalent to the following statement:

			\item there exists some $K\ge 1$ such that
			$$
				\limsup_{t\to\infty} {\widetilde{B}(t)\over \widetilde{B}(Kt)} < K^{{1\over \alpha -1}}.
			$$
		\end{enumerate}
		\end{theoremA}

Note that the condition on $\varphi$ causes no loss of generality, since
otherwise $H_\alpha^\beta\colon L^1(0,1) \to M(0,1)$. The details are discussed
in Remark~\ref{rem:wlog}.

The proof of Theorem~A relies on the next result of independent interest, which provides us with a
reduction theorem for Orlicz and Marcinkiewicz spaces.

	\begin{theoremB} \label{thm:Reduction_Orl_Marc}
		Let $0<\alpha<1$, $\beta>0$, $\alpha+1/\beta\ge 1$
		and let $L^A(\domain)$ be an Orlicz space with a
		Young function $A$ and $M(\domain)$ be a Marcinkiewicz endpoint space
		with a fundamental function $\varphi$ satisfying
		$$
			\sup_{0<t<1} \varphi(t^{1\over\beta})\, t^{{\alpha}-1}=\infty.
		$$
		Then the relation
		$$
			H_\alpha^\beta\colon L^A(\domain) \to M(\domain)
		$$
		holds if and only if there exists $C>0$ such that
		$$
			\int_1^t {\widetilde{A}(s)\over s^{{1/(1-\alpha)}+1}}\,\d s
				\lesssim {\widetilde{B}(Ct)\over t^{1/(1-\alpha)}},
					\quad t\in(2,\infty),
		$$
		where $B$ is a Young function described in Theorem~A
		and $\widetilde{A}$ and $\widetilde{B}$ are complementary Young functions to $A$ and $B$
		respectively.
		\end{theoremB}

Our final principal result describes the fundamental function of the optimal r.i.\ domain space.

	\begin{theoremC} \label{thm:Fund_of_domain}
		Let $0<\alpha<1$, $\beta>0$, $\alpha+1/\beta\ge 1$.
		Suppose that $M(\domain)$ is the Marcinkiewicz endpoint space with fundamental function $\varphi$.
		Then the fundamental function $\varphi_X$ of the largest r.i.\ space $X(\domain)$ having the property
		$$
			H_{\alpha}^{\beta}\colon X(\domain)\to M(\domain)
		$$
		satisfies
		$$
			\varphi_X(t)\simeq t\sup_{t<s<1} \varphi(s^{1\over\beta})\,s^{\alpha-1},
				\quad t\in (0,1).
		$$
		\end{theoremC}

The paper is structured as follows. In Section~\ref{sec:fs} we collect all the
necessary basic background material. In Section~\ref{sec:BC} we prove
Theorem~B and Theorem~C. In
Section~\ref{sec:main} we prove Theorem~A. Finally,
Section~\ref{sec:app} contains various applications and examples of the main
result.

\section{Function spaces}\label{sec:fs}

Let us now fix the notation which will be used in this paper.

By $A \lesssim B$ and $A \gtrsim B$ we mean that $A \le C\,B$ and $A \ge C\,B$,
respectively, where $C$ is a positive constant independent of the appropriate
quantities involved in $A$ and $B$. We shall write $A \simeq B$ when both of
the estimates $A \lesssim B$ and $A \gtrsim B$ are satisfied. We shall use
the convention $0\cdot\infty=0$, ${0\over 0}=0$ and
${\infty\over\infty}=0$.

When $X$ and $Y$ are Banach spaces, we say that $X$ is \df{embedded} into $Y$, and
write $X\embed Y$, if $X \subseteq Y$ and there exists a positive constant $C$,
such that $\|f\|_Y \le C\, \|f\|_X$ for every $f\in X$.
We say that a linear operator $T$ defined on $X$ with values in $Y$ is \df{bounded}
if there exists a constant $C>0$ such that $\|Tf\|_Y \le C \|f\|_X$ for
every $f\in X$. We write $T\colon X\to Y$ in this case.

We say that a function $G\colon[0,\infty)\to(0,\infty)$ satisfies the
\df{$\Delta_2$ condition at infinity} if there exist $K>0$ and $T\ge 0$ such that $G(2t)\le
K\,G(t)$ for every $t\ge T$. We will use only $\Delta_2$ condition at infinity,
hence we shall shortly say $\Delta_2$ condition and write $G\in\Delta_2$. 

For a nonnegative function
$f$ we shall write $\int_0 f < \infty$ when there exists some $c>0$ such that
the integral $\int_0^c f$ converges.  By integral we always mean the Lebesgue
integral.

\subsection{Rearrangement-invariant spaces}

In this section we recall definitions and some basic facts concerning the
re\-arrange\-ment-invariant spaces, which we will need in the following text. We
shall not prove well-known results; all proofs and further details can be found in the
monograph by C.\,Bennett and R.\,Sharpley \cite{BS}.

Suppose $\Omega$ is a domain in $\R^n$. Let $\mathcal{M}(\Omega)$
be a class of real-valued measurable functions on $\Omega$ and $\mathcal{M}^+(\Omega)$
the class of nonnegative functions in $\mathcal{M}(\Omega)$. Given
$f\in\mathcal{M}$ we define its \df{nonincreasing rearrangement} on $(0,|\Omega|)$
as
$$
	f^*(t):=\inf\bigl\{\lambda>0,\;\mu_f(\lambda)\le t\bigr\},\quad 0<t<|\Omega|,
$$
where $\mu_f$ is the \df{distribution function} of $f$, i.e.,
$$
	\mu_f(\lambda):=\bigl|\bigl\{x\in\Omega,\;|f(x)|>\lambda\bigr\}\bigr|,\quad \lambda>0,
$$
where the $|\cdot|$ stands for the Lebesgue measure.
The \df{Hardy average} $f^{**}$ is defined on $(0,|\Omega|)$ as
$$
	f^{**}(t)={1\over t}\int_0^t f^*(s)\,\d s,\quad 0<t<|\Omega|.
$$
Let $f$, $g\in\mathcal{M}^+(\Omega)$. Then we have the \df{Hardy-Littlewood inequality}
$$
	\int_\Omega f(x)\,g(x)\,\d x \le \int_0^{|\Omega|} f^*(t)\,g^*(t)\,\d t.
$$

When $E\subseteq \Omega$ is measurable, we denote by $\chi_E$ the \df{characteristic
function} of $E$.
A \df{simple function} is a finite sum $\sum_j \lambda_j\chi_{E_j}$,
where $\lambda_j\ne 0$ is a real number and $E_j\subseteq\Omega$ has finite measure for
every index $j$.

Denote by $I$ the interval $(0,1)$.
A mapping $\rho\colon\mathcal{M}^+(I)\to[0,\infty]$ is called a
\df{re\-arr\-ange\-ment-invariant (r.i.)\ Banach function norm } on $\mathcal{M}^+(I)$, if for
all $f$, $g$, $f_n$ ($n\in\mathbb{N}$) in $\mathcal{M}^+(I)$, for all constants $a\ge 0$ and
for every measurable subset $E$ of $I$, the following properties hold:
	\makeatletter
	\tagsleft@true
	\makeatother
\begin{align*}
	& \rho(f) = 0 \;\leftrightarrow\; f=0\;\hbox{\fam0 a.e.};\;
	  \rho(af) = a\rho(f);\;
	  \rho(f+g) \le \rho(f)+\rho(g);\tag{P1}\\
	&0\le f\le g\;\hbox{\fam0 a.e.} \;\text{implies}\; \rho(f)\le\rho(g);\tag{P2}\\
	&0\le f_n\upto f\;\hbox{\fam0 a.e.} \;\text{implies}\; \rho(f_n)\upto\rho(f);\tag{P3}\\
	&\rho(\chi_I)<\infty;\tag{P4}\\
	&\textstyle\int_0^1f(x)\,\d x\lesssim \rho(f);\tag{P5}\\
	&\rho(f)=\rho(f^*).\tag{P6}
\end{align*}
	\makeatletter
	\tagsleft@false
	\makeatother
The \df{associate norm} of an r.i.\ norm $\rho$ is another such norm $\rho'$
defined as
$$
	\rho'(g):=\sup_{\rho(f)\le 1}\;\int_0^1 g(t)\,f(t)\,\d t,\quad f,g\in\mathcal{M}^+(I).
$$
It obeys the \df{Principle of Duality}; that is,
$$
	\rho'':=(\rho')'=\rho.
$$
Furthermore, the \df{H\"older inequality}
$$
	\int_0^1 f(t)\,g(t)\,\d t\le \rho(f)\,\rho'(g)
$$
holds for every $f,g\in\mathcal{M}^+(I)$.

Given the r.i.\ norm $\rho$, the corresponding \df{rearrangement-invariant Banach function space}
or, for short, \df{r.i.\ space}, is the collection
$$
	L_\rho(I):=\bigl\{f\in\mathcal{M}(I),\;\rho(|f|)<\infty\bigr\},
$$
endowed with the norm
$$
	\|f\|_{L_\rho(I)} := \rho(|f|),\quad f\in L_\rho(I).
$$
Next, given a bounded domain $\Omega$ in $\R^n$, we define the r.i.\ space
$$
	L_\rho(\Omega):=\bigl\{f\in\mathcal{M}(\Omega),\;\rho\bigl(f^*(t|\Omega|)\bigr)<\infty\bigr\}
$$
with
$$
	\|f\|_{L_\rho(\Omega)}:=\rho\bigl(f^*(t|\Omega|)\bigr),\quad f\in L_\rho(\Omega).
$$
If $\rho_1$ and $\rho_2$ are two r.i.\ norms, then
$L_{\rho_1}(\Omega)\subseteq L_{\rho_2}(\Omega)$
implies $L_{\rho_1}(\Omega)\embed L_{\rho_2}(\Omega)$.

Let $\varphi$ be a nonnegative function defined on the interval $[0,\infty)$.
If
\begin{enumerate}[(i)]
	\item $\varphi(t)=0$ iff $t=0$,
	\item $\varphi(t)$ is nondecreasing on $(0,\infty)$,
	\item $\varphi(t)/t$ is nonincreasing on $(0,\infty)$,
\end{enumerate}
then $\varphi$ is said to be \df{quasiconcave}.
We also say that a function $\varphi$ defined on bounded interval $[0,R]$, for
$R\in(0,\infty)$, is
quasiconcave if the continuation by constant value $\varphi(R)$ is quasiconcave on $[0,\infty)$.

The \df{fundamental function} of an r.i.\ norm $\rho$ on $\mathcal{M}^+(I)$ is defined
by
$$
	\varphi_\rho(t):=\rho\bigl(\chi_{(0,t)}\bigr),\quad t\in I,\quad \varphi_\rho(0)=0.
$$
The fundamental function is quasiconcave on $[0,1)$, continuous except perhaps at the origin and satisfies
$$
	\varphi_\rho(t)\,\varphi_{\rho'}(t)=t,\quad t\in I.
$$
Quasiconcave functions need not be concave, however, every r.i.\ space can be equivalently
renormed so that its fundamental function is concave.

Let $\varphi$ be a concave function. We define the \df{Lorentz endpoint space}
$\Lambda_\varphi(\Omega)$ by the function norm
$$
	\rho_{\Lambda_\varphi}(f):=\int_0^1f^*(t)\,\d\varphi(t),
		\quad f\in\mathcal{M}^+(I),
$$
where $\d\varphi$ stands for the Lebesgue-Stieltjes measure associated with $\varphi$.
We define the \df{Marcinkiewicz endpoint space} $M_\varphi(\Omega)$ by the function norm
$$
	\rho_{M_\varphi}(f):=\sup_{0<t<1}f^{**}(t)\,\varphi(t),
		\quad f\in\mathcal{M}^+(I).
$$
The endpoint spaces $\Lambda_\varphi(\Omega)$ and $M_\varphi(\Omega)$ are r.i.\ spaces with
the fundamental function $\varphi$. If $X(\Omega)$ is an r.i.\ space with the fundamental function
$\varphi$, then
$$
	\Lambda_\varphi(\Omega)\embed X\embed M_\varphi(\Omega).
$$
In other words, $\Lambda_\varphi(\Omega)$ and $M_\varphi(\Omega)$
are respectively the smallest and the largest r.i.\ spaces having the fundamental
function equivalent to $\varphi$.

The associate space of a Lorentz endpoint space $\Lambda_\varphi$
is the Marcinkiewicz endpoint space $M_{\psi}$ where
both $\varphi$ and $\psi$ are concave and $\varphi(t)\,\psi(t)=t$ on $I$.

If $|\Omega|<\infty$, then for every r.i.\ space $X(\Omega)$
$$
	L^\infty(\Omega) \embed X(\Omega) \embed L^1(\Omega).
$$

Assume either $1< p,q <\infty$ or $p=q=1$ or $p=q=\infty$.  The \df{Lorentz
space} $L^{p,q}(\Omega)$ is defined by the functional
$$
	\rho_{{p,q}}(f)=\rho_q \left(t^{{1\over p}-{1\over q}}f^*(t)\right),
		\quad f\in\mathcal{M}^+(I),
$$
where
$$
	\rho_q(f)=\begin{cases}
		\left(\displaystyle \int_0^1\, f(t)^q\,\d t\right)^{1\over q},
			& 1\le q < \infty,\\
		\displaystyle \esssup_{0< t< 1}\, f(t),
			& q=\infty,
	\end{cases}
$$
stands for the Banach function norm of the \df{Lebesgue space} $L^q(\Omega)$.
The functional $\rho_{{p,q}}$ is a Banach function norm if and only
if $1\le q\le p$. However, for $1<p<\infty$, $\rho_{{p,q}}$ can be
equivalently replaced by Banach function norm
$$
	\rho_{{(p,q)}}(f)=\rho_q \left(t^{{1\over p}-{1\over q}}f^{**}(t)\right).
$$
The fundamental function of the norm $\rho_{(p,q)}$ satisfies
$$
	\varphi_{\rho_{(p,q)}}(t)\simeq t^{1\over p},
		\quad t\in [0,1).
$$
The spaces $L^{p,1}(\Omega)$ and $L^{p,\infty}(\Omega)$ are equal to the Lorentz and Marcinkiewicz
endpoint spaces $\Lambda_\varphi(\Omega)$ and $M_\varphi(\Omega)$, respectively, with $\varphi(t)=t^{1/p}$.
If the first parameter is fixed, then the Lorentz spaces are nested, i.e., we
have $L^{p,q}(\Omega)\embed L^{p,r}(\Omega)$ whenever $1<p<\infty$ and $1\le q\le r\le \infty$.

\subsection{Orlicz Spaces}

We also need to know definitions and all the basic facts about Young
functions and Orlicz Spaces. All of these can be found for instance in the book by
L.\,Pick, A.\,Kufner, O.\,John and S.\,Fu\v c\'\i k \cite{FS}.

We shall say that $A$ is a \df{Young function} if there exists a function $a\colon[0,\infty)\to[0,\infty)$
such that
$$
	A(t)=\int_0^t a(s)\,\d s,
		\quad t\in [0,\infty),
$$
and $a$ has the following properties:
\begin{enumerate}[(i)]
	\item $a(s)>0$ for $s>0$, $a(0)=0$; 
	\item $a$ is right-continuous;
	\item $a$ is nondecreasing;
	\item $\lim_{s\to\infty} a(s)=\infty$.
\end{enumerate}
Every Young function is continuous, nonnegative, strictly increasing, convex on $[0,\infty)$
and satisfies
$$
	\lim_{t\to 0^+}{A(t)\over t}=
	\lim_{t\to\infty}{t\over A(t)}=0.
$$
Furthermore, one has
$$
	A(\alpha t) \le \alpha \,A(t),\quad \alpha\in [0,1],\quad  t\ge 0,
$$
and
$$
	A(\beta t) \ge \beta\, A(t),\quad \beta\in (1,\infty),\quad  t\ge 0.
$$
Moreover $A(t)/t$ is increasing on $(0,\infty)$ and we have the estimates
$$
	A(t) \le a(t)\,t \le A(2t),
		\quad t\in(0,\infty).
$$

For a Young function $A$ and a domain $\Omega\subseteq \R^n$, the \df{Orlicz
space} $L^A=L^A(\Omega)$ is the collection of all functions $f\in\mathcal{M}(\Omega)$
for which there exists a $\lambda>0$ such that
$$
	\int_{\Omega} A\left({|f(x)|\over \lambda}\right)\,\d x <\infty.
$$
The Orlicz space $L^A(\Omega)$ is endowed with the \df{Luxemburg norm}
$$
	\|f\|_{L^A}:=\inf\left\{\lambda>0,\;
		\int_{\Omega}A\left({|f(x)|\over \lambda}\right)\,\d x \le 1\right\}.
$$

The \df{complementary function} $\widetilde{A}$ of a Young function $A$ is given
by
$$
	\widetilde{A}(t):=\sup_{s>0}\bigl(st-A(s)\bigr),\quad t\in [0,\infty).
$$
The complementary function $\widetilde{A}$ is a Young function as well and
the complementary function of $\widetilde{A}$ is once more $A$.
For any Young function $A$ and its complementary function $\widetilde{A}$ there
is the relation
$$
	t\le A^{-1}(t)\,\widetilde{A}^{-1}(t)\le 2t,\quad t\in[0,\infty).
$$

With the help of the complementary function, we can define an alternative 
\df{Orlicz norm} on an Orlicz space by
$$
	\|f\|_{(L^A)}:=\sup\left\{\int_\Omega \bigl|f(x)\,g(x)\bigr|\,\d x\right\},
$$
where the supremum is taken over all functions $g\in\mathcal{M}(\Omega)$ such that
$$
	\int_{\Omega} \widetilde{A}\bigl(|g(x)|\bigr)\,\d x \le 1.
$$
The Luxemburg and Orlicz norms are equivalent, namely,
$$
	\|f\|_{L^A}\le \|f\|_{(L^A)}\le 2\,\|f\|_{L^A}.
$$

When $L^A(\Omega)$ is an Orlicz space
endowed with the Luxemburg norm, then the associate space is $L^{\widetilde{A}}(\Omega)$
with the Orlicz norm.
In particular, the sharp H\"older inequality for Orlicz spaces has the form
$$
	\int_{\Omega} \bigl|f(x)\,g(x)\bigr|\,\d x \le \|f\|_{L^A}\,\|f\|_{(L^{\widetilde{A}})}.
$$

The Orlicz space $L^A(\Omega)$ is an r.i.\ space and
$$
		\|\chi_E\|_{L^A} = {1\over A^{-1}\bigl({\textstyle {1\over |E|}}\bigr)}
$$
for every measurable $E\subseteq\Omega$ of positive measure,
thus, for a bounded domain $\Omega$, the fundamental function
for the Luxemburg norm is
$$
	\varphi_{L^A}(t) = {1\over A^{-1}\bigl({1\over t|\Omega|}\bigr)},
		\quad t\in I,\quad \varphi_{L^A}(0)=0.
$$

An Orlicz space $L^A(I)$ with fundamental function $\varphi$ coincides with
the Marcinkiewicz endpoint space $M_\varphi(I)$ if and only if there exists $\delta\in(0,1)$
such that
\begin{equation}
	\int_0^1 A\bigl(\delta\, A^{-1}({\textstyle {1\over t}})\bigr)\,\d t < \infty
	\label{eq:orlismarc}
\end{equation}
(see also~\cite{L}).

For $|\Omega|<\infty$, the inclusion relation between Orlicz spaces is governed
by inequalities involving the corresponding Young functions. If $A$ and $B$ are
Young functions then $L^A(\Omega)\embed L^B(\Omega)$ if and only if there exist
$c>0$ and $T\ge 0$ such that
$$
	B(t)\le A(ct),\quad t\ge T,
$$
which we denote by $B\prec A$ or $A\succ B$. If both $A\prec B$ and
$A\succ B$ hold, we say that $A$ and $B$ are equivalent and write $A\approx B$.
When $|\Omega|<\infty$, the inclusion $L^A(\Omega)\subseteq L^B(\Omega)$ is proper if and only
if
$$
	\limsup_{t\to\infty}{B(t)\over A(\lambda t)}=0
$$
for every $\lambda>0$.
In such case we write $B\eprec A$ or $A\esucc B$.

If $A \prec B$ or $A\eprec B$ then $\widetilde{A}\succ\widetilde{B}$ or
$\widetilde{A}\esucc\widetilde{B}$, respectively.

\section{Proofs of Theorems~B and C}
	\label{sec:BC}

	\begin{lemma}
		\label{lemm:function_E_properties_i}
		Let $A$ be a Young function and let $\xi$ be a nonzero real number.
		Assuming
		\begin{equation}
			\int_0 A(s)\,s^{{1\over\xi}-1}\,\d s < \infty,
			\label{eq:converges_near_zero}
		\end{equation}
		we define
		$$
			E_\xi(t)=|\xi|^{-1}t^{-{1\over\xi}}\int_0^t A(s)\,s^{{1\over\xi}-1}\,\d s,\quad t\in(0,\infty).
		$$
		Such $E_\xi$ is an increasing mapping of
		$(0,\infty)$ onto itself.  Moreover, if $R\in(0,\infty]$, then the following
		relations hold.
		\begin{align}
			\|t^\xi\chi_{(0,a)}(t)\|_{L^A(0,R)}
				&= {a^\xi\over E_\xi^{-1}\left({1\over a}\right)},
					\quad a\in(0,R),\; \xi>0,
					\label{eq:Function_Ei}
					\\
			\|t^\xi\chi_{(a,\infty)}(t)\|_{L^A(0,\infty)}
				&= {a^\xi\over E_\xi^{-1}\left({1\over a}\right)},
					\quad a\in(0,\infty),\; \xi<0.
					\label{eq:Function_Eii}
		\end{align}
		If, in addition, $\varepsilon\in(0,R)$ and if $\xi<0$ then
		\begin{equation}
			\|t^\xi\chi_{(a,R)}(t)\|_{L^A(0,R)}\simeq
				\|t^\xi\chi_{(a,\infty)}(t)\|_{L^A(0,\infty)},
					\quad a\in(0,R-\varepsilon).
				\label{eq:Function_Eiii}
		\end{equation}
		\end{lemma}

	\begin{proof} Assume \eqref{eq:converges_near_zero}. By change of variables	$s\mapsto ts$ we have
		$$
			E_{\xi}(t) =|\xi|^{-1} \int_0^1 A(ts)\,s^{{1\over \xi}-1}\,\d s,
				\quad t\in(0,\infty),
		$$
		hence $E_\xi$ is increasing.

		By definition of the Luxemburg norm, we have
		$$
		\|t^\xi \chi_{(0,a)}(t)\|_{L^A(0,R)}
			=\inf\biggl\{\lambda>0,\; \int_0^a A\left({t^\xi\over \lambda}\right)\,\d t\le 1 \biggr\}.
		$$
		Next, by change of variables we get for $\xi>0$
		\begin{align*}
			\|t^\xi \chi_{(0,a)}(t)\|_{L^A(0,R)}
				&= \inf\biggl\{\lambda>0,\; {\lambda^{1\over \xi}\over \xi}
					\int_0^{{a^\xi\over \lambda}} A(s)\,s^{{1\over \xi}-1}\,\d s\le 1 \biggr\} \\
				&= \inf\Bigl\{\lambda>0,\; a\,E_\xi \bigl({\textstyle {a^\xi\over \lambda}}\bigr) \le 1 \Bigr\} \\
				&= {a^\xi \over E^{-1}_\xi\left({1\over a}\right)}.
		\end{align*}
		This proves the part \eqref{eq:Function_Ei}. The proof of the relation \eqref{eq:Function_Eii} can
		be done in an analogous way and we omit it.
			
		It remains to prove the \eqref{eq:Function_Eiii}.
		Clearly,
		$$
			\|t^\xi\chi_{(a,\infty)}(t)\|_{L^A(0,\infty)}
				\ge \|t^\xi\chi_{(a,R)}(t)\|_{L^A(0,\infty)}
				= \|t^\xi\chi_{(a,R)}(t)\|_{L^A(0,R)}
		$$
		by the monotonicity of the norm.
		On the other hand, we have by the triangle inequality
		$$
			\|t^\xi\chi_{(a,\infty)}(t)\|_{L^A(0,\infty)}
				\le \|t^\xi\chi_{(a,R)}(t)\|_{L^A(0,R)}
				+ \|t^\xi\chi_{(R,\infty)}(t)\|_{L^A(0,\infty)}.
		$$
		Using \eqref{eq:Function_Eii}, the term $\|t^\xi\chi_{(R,\infty)}(t)\|_{L^A(0,\infty)}$ equals
		$R^\xi\!/E^{-1}_\xi({1/R})$ since $\xi<0$.
		Thanks to the assumptions, this quantity is finite, say $K$.
		The term $\|t^\xi\chi_{(a,R)}(t)\|_{L^A(0,R)}$ is a decreasing
		function of the variable $a$, positive on $(0,R)$ and vanishing at $R$.
		Hence for every $\varepsilon\in(0,R)$ there exists a constant $C$
		such that
		$$
			K\le C\, \|t^\xi\chi_{(a,R)}(t)\|_{L^A(0,R)},\quad a\in(0,R-\varepsilon).
		$$
		For those $a$ we conclude that
		$$
			\|t^\xi\chi_{(a,\infty)}(t)\|_{L^A(0,\infty)}
				\le (C+1)\, \|t^\xi\chi_{(a,R)}(t)\|_{L^A(0,R)}.
		$$
		\end{proof}

	\begin{lemma} \label{lemm:quasi}
		Let $0<\alpha<1$, $\beta>0$, $\alpha+1/\beta\ge 1$
		and let $\varphi$ be a quasiconcave function on $(0,1)$. We define
		$$
			\bar{\varphi}(t)=
				t^{\beta(1 - \alpha)}\sup_{t<s<1} \varphi(s)\,s^{\beta(\alpha - 1)},
					\quad t\in (0,1),\quad \bar\varphi(0)=0.
		$$
		Then $\bar{\varphi}(t)$ and $\bar{\varphi}(t^{1/\beta})\,t^{\alpha}$ are quasiconcave.
		\end{lemma}

	\begin{proof}
		Since $\varphi$ is nondecreasing, we have for every $t\in(0,1)$
		\begin{align*}
				\bar{\varphi}(t)
					&= t^{\beta(1-\alpha)}\sup_{t<s<1 } s^{\beta(\alpha-1)} \sup_{0<r<s} \varphi(r)\\
					&= t^{\beta(1-\alpha)}\sup_{0<r<1} \varphi(r) \sup_{\max\{r,t\}<s<1} s^{\beta(\alpha-1)}\\
					&= t^{\beta(1-\alpha)}\sup_{0<r<1} \varphi(r) \,\min\left\{t^{\beta(\alpha-1)},r^{\beta(\alpha-1)}\right\}\\
					&= \sup_{0<r<1} \varphi(r)\, \min\left\{1,\left({\textstyle {t\over r}}\right)^{\beta(1-\alpha)}\right\},
		\end{align*}
		hence $\bar{\varphi}$ is nondecreasing. Next, by definition, we have
		$$
			{\bar{\varphi}(t)\over t}
				= t^{\beta(1-\alpha)-1}\sup_{t<s<1} \varphi(s)\,s^{\beta(\alpha-1)},
				\quad t\in(0,1).
		$$
		The function $t^{\beta(1-\alpha)-1}$ is nonincreasing
		since the exponent $\beta(\alpha-1)+1$ is nonnegative by the assumptions of the lemma.
		Hence ${\bar{\varphi}(t)/ t}$ is nonincreasing on $(0,1)$.

		The function $\bar{\varphi}(t^{1/\beta})\,t^{\alpha}$ is increasing
		as a composition of nondecreasing functions multiplied by increasing function~$t^\alpha$.
		Next the expression 
		$$
			{\bar{\varphi}(t^{1\over\beta})\,t^{\alpha}\over t}=\sup_{t<s<1} \varphi(s^{1\over\beta})\,s^{\alpha-1},
				\quad t\in(0,1),
		$$
		is nonincreasing. The rest is trivial.
		\end{proof}

	\begin{lemma} \label{lemm:YoungB}
		Let $u$ be a quasiconcave, right continuous at origin and strictly increasing function on $[0,1)$
		such that
		$$
			\lim_{t\to 0^+} {u(t)\over t} =\infty.
		$$
		Then there exists a Young function $B$ such that
		the fundamental function of the Orlicz space $L^{B}(0,1)$ is equivalent to
		$u$ on $[0,1)$.
		Moreover
		$$
			\widetilde{B}^{-1}(t) \simeq 
				t\, u\left({\textstyle {1\over t}}\right)
				\quad t\in (1,\infty),
		$$
		where $\widetilde{B}$ is the complementary Young function to $B$.
		\end{lemma}

	\begin{proof}
		We can assume without loss of generality that $u(1)=1$. Then by continuity of $u$ we have $u(0,1)=(0,1)$.
		Let us define
		$$
			{b}(s)=\begin{cases}
			{1\over s\, u^{-1}\left({{1\over s}}\right) }, &\quad s\in(1,\infty), \\
			s, &\quad s\in[0,1].
			\end{cases}
		$$
		and
		$$
			{B}(t)=\int_0^t{b}(s)\,\d s,
				\quad t\in[0,\infty).
		$$
		We claim that ${B}$ is a Young function. The properties (i) and (ii) from the
		definition of Young function are clear. Let us prove that $b$ is nondecreasing.
		The function $u(t) / t$ is nonincreasing and $u$ itself is increasing,
		hence $s/u^{-1}(s)$ is nonincreasing
		and therefore ${b}(s)={1\over s\, u^{-1}(1/s)}$ is nondecreasing on $(1,\infty)$ and
		also (trivially) on $[0,1]$.
		It remains to show that $\lim_{s\to\infty}{b}(s)=\infty$.
		Indeed,
		$$
			\lim_{s\to\infty} b(s)
				= \lim_{t\to 0^+} {t\over u^{-1}(t)}
				= \lim_{t\to 0^+} {u(t)\over t}
				= \infty.
		$$

		Now, since $B$ is a Young function, we have that
		$$
			{B}(t) \le {b}(t)\, t \le {B}(2t),
				\quad t\in[0,\infty).
		$$
		It follows by definition of ${b}$ that
		$$
			B(t) \le {1\over u^{-1}\left({1\over t}\right)} \le B(2t),
				\quad t\in (1,\infty).
		$$
		Applying the increasing function $B^{-1}$, we get
		$$
			t \le B^{-1}\left({1\over u^{-1}\left({1\over t}\right)}\right) \le 2t,
				\quad t\in (1,\infty),
		$$
		that is, taking reciprocal values and $t\mapsto 1/s$,
		$$
			{s\over 2}\le {1\over {B}^{-1}\left({\textstyle {1\over u^{-1}(s)}}\right)}
				\le s,\quad s\in(0,1).
		$$
		Finally, since $u$ is increasing on $(0,1)$ and $u(0,1)=(0,1)$,
		this implies
		$$
			{u(y)\over 2}\le {1\over {B}^{-1}\bigl({\textstyle {1\over y}}\bigr)}
				\le u(y),\quad y\in(0,1).
		$$
		Hence by the definition of the fundamental function for the Luxemburg norm we conclude that
		$$
			\varphi_{L^{{B}}}(t)\simeq u(t),
				\quad t\in (0,1).
		$$
		In addition
		$$
			\widetilde{B}^{-1}(t) \simeq {t\over {B}^{-1}(t)}
				\simeq t\, u\left({\textstyle {1\over t}}\right),
					\quad t\in(0,1).
		$$
		The proof is complete.
		\end{proof}

	The following proposition enables us to reduce an embedding to a Lorentz endpoint spaces
	only to testing on characteristic functions.
	The idea of this statement is based on \citep[Theorem 7]{C},
	where the Lorentz space $L^{p,1}(\Omega)$ occurs as a target space,
	nonetheless the proof also works for any Lorentz endpoint space.
	For the sake of completeness, we show also the proof here.

	\begin{proposition} \label{prop:Lambda_space_embedding}
		Let $Y(\domain)$ be a
		Banach function space and $\Lambda(\domain)$ be a Lorentz endpoint space over $(\domain)$.
		Suppose that $T$ is a sublinear operator mapping $\Lambda(\domain)$ to $Y(\domain)$
		and satisfying
		\begin{equation}
			\|T\chi_E\|_{Y(\domain)} \lesssim \|\chi_E\|_{\Lambda(\domain)}
				\label{eq:ineq_on_char}
		\end{equation}
		for every measurable set $E\subseteq (\domain)$. Then 
		$$
			\|Tf\|_{Y(\domain)} \lesssim \|f\|_{\Lambda(\domain)}
		$$
		for every $f\in\Lambda(\domain)$.
		\end{proposition}

	\begin{proof} Let $f$ be a simple nonnegative function on $(\domain)$. 
		Thus $f$ can be written as a finite sum $f=\sum_j \lambda_j\chi_{E_j}$,
		where $\lambda_j$ are positive real numbers and the sets $E_j$ are measurable
		subsets of $(\domain)$ satisfying $E_1\subseteq E_2\subseteq\cdots$.
		Then, as readily seen, we have $f^*=\sum_j \lambda_j\chi^*_{E_j}$.
		Let $\varphi$ be a fundamental function of $\Lambda(\domain)$. By the definition of the Lorentz norm we have
		$$
			\|f\|_{\Lambda(\domain)}
				= \int_0^1f^*\,\d\varphi
				= \int_0^1 \sum_j \lambda_j\chi^*_{E_j}\,\d\varphi
				= \sum_j \lambda_j\int_0^1 \chi^*_{E_j}\,\d\varphi
				= \sum_j \lambda_j \|\chi_{E_j}\|_{\Lambda(\domain)}.
		$$
		On account of the sublinearity of $T$ we have
		$|Tf| \le \sum_j \lambda_j|T\chi_{E_j}|$,
		and consequently by~\eqref{eq:ineq_on_char} and by axioms (P1) and (P2)
		we obtain
		$$
			\|Tf\|_{Y(\domain)}
				\le \sum_j\lambda_j\|T\chi_{E_j}\|_{Y(\domain)}
				\lesssim \sum_j\lambda_j\|\chi_{E_j}\|_{\Lambda(\domain)}
				= \|f\|_{\Lambda(\domain)}.
		$$
		Now if $f$ is simple but no longer nonnegative, we use
		the same for the positive part of $f$ and for the negative part of~$f$.

		Suppose that $f$ is an arbitrary function in $\Lambda(\domain)$ and let $f_n$ be a sequence
		of simple integrable functions converging to $f$ in $\Lambda(\domain)$.
		Then
		$$
			\|T(f_n) - T(f_m)\|_{Y(\domain)}
				\le \|T(f_n - f_m)\|_{Y(\domain)}
				\lesssim \|f_n - f_m\|_{\Lambda(\domain)},
		$$
		and $Tf_n$ is Cauchy, hence convergent in $Y(\domain)$. Since limits are unique in $Y(\domain)$,
		it follows that $\lim Tf_n=Tf$ and
		$$
			\|Tf\|_{Y(\domain)}
				= \lim \|Tf_n\|_{Y(\domain)}
				\lesssim \lim \|f_n\|_{\Lambda(\domain)}
				= \|f\|_{\Lambda(\domain)}
		$$
		as we wished to show.
		\end{proof}

	Next proposition provides the optimal r.i.\ range space for
	the operator $H_\alpha^\beta$ and a given r.i.\ domain space.
	The proof can be obtained by simple modification of the proof
	of \citep[Theorem 4.5]{EKP}, where $\beta=1$ and $\alpha=1/n$
	and therefore is omitted.

	\begin{proposition} \label{prop:opt_ri_range}
		Let $X(0,1)$ be an r.i.\ space, 
		$0<\alpha<1$, $\beta>0$ and $\alpha+1/\beta\ge 1$.
		Then
		$$
			Y'(0,1) := \Bigl\{ f\in\mathcal{M}(0,1),\;
				\|f\|_{Y'(0,1)} := \bigl\| \bigl(H_\alpha^\beta\bigr)' f^* \bigr\|_{X'(0,1)} < \infty
			\Bigr\}
		$$
		is an r.i.\ space, such that the associate space $Y(0,1)$ is the smallest space
		among r.i.\ spaces rendering $H_\alpha^\beta\colon X(0,1) \to Y(0,1)$ true.
		\end{proposition}

	The construction of the optimal r.i.\ domain for $H_\alpha^\beta$ and
	a given r.i.\ range space is similar to that in \citep[Theorem 3.3]{KP},
	as well as its proof, needing only trivial modifications.
	The fact that $\mu_f=\mu_h$ is denoted by $f\sim h$.

	\begin{proposition} \label{prop:opt_ri_domain}
		Let $Y(0,1)$ be an r.i.\ space
		such that $Y(0,1)\embed L^{{1\over \beta(1-\alpha)},1}(0,1)$.
		Then
		$$
			X(0,1) := \Bigl\{
				f\in\mathcal{M}(0,1),\; \|f\|_{X(0,1)} := \sup_{h\sim f}\, \bigl\|H_\alpha^\beta\, h \bigr\|_{Y(0,1)} < \infty
				\Bigr\}
		$$
		is the largest r.i.\ space satisfying $H_\alpha^\beta\colon X(0,1) \to Y(0,1)$.
		\end{proposition}

	\begin{remark} \label{rem:wlog}
		Now, under the same assumptions on $\alpha$ and $\beta$ as in Proposition~\ref{prop:opt_ri_range}
		one can readily calculate the optimal endpoint estimates
		\begin{equation}
			H_\alpha^\beta\colon L^1(0,1) \to L^{{1\over\beta(1-\alpha)},1}(0,1)
				\label{eq:opt_emb_a}
		\end{equation}
		and
		\begin{equation}
			H_\alpha^\beta\colon L^{{1\over\alpha},1}(0,1) \to L^{\infty}(0,1).
				\label{eq:opt_emb_b}
		\end{equation}
		The relation \eqref{eq:opt_emb_a} shows that the assumption in Proposition~\ref{prop:opt_ri_domain}
		cause no loss of generality.

		Let us also discuss the assumption on fundamental function
		$$
			\sup_{0<t<1} \varphi(t^{1\over\beta})\, t^{{\alpha}-1}=\infty
		$$
		in Theorems~A and B.
		If this condition is not satisfied, then
		$$
			\varphi(t) \le C t^{\beta(1-\alpha)},
				\quad t\in(0,1),
		$$
		for some $C>0$, which is equivalent to
		$$
			L^{{1\over \beta(1-\alpha)},\infty}(0,1) \subseteq M_\varphi(0,1),
		$$
		hence, thanks to \eqref{eq:opt_emb_a}, also to
		$$
			H_\alpha^\beta\colon L^1(0,1) \to M_\varphi(0,1).
		$$
		Since $L^1(0,1)$ is the largest r.i.\ space, we can see that this considered
		assumption cause no relevant restriction to target spaces.
	\end{remark}

	\begin{proof}[Proof of Theorem~C]
		We first prove the inequality ``$\gtrsim$''.
		Let $\alpha$ and $\beta$ be as in the theorem and let us set
		$$
			\psi(t)=t \sup_{t < s < 1}\varphi(s^{1\over\beta})\,s^{\alpha -1},
				\quad t\in(0,1).
		$$
		Then, by Lemma~\ref{lemm:quasi}, $\psi(t)$ is quasiconcave function
		on $(0,1)$ and $\psi(t)\ge t^{\alpha}\,\varphi(t^{1/\beta})$ for $t\in(0,1)$.
		We claim that $\psi$ is up to equivalence the smallest function with this property.
		Indeed,
		let $\eta(t)$ be a quasiconcave function on $[0,1)$
		and $\eta(t)\ge t^{\alpha}\,\varphi(t^{1/\beta})$ for $t\in(0,1)$. Then
		$$\displaylines{
			\varphi(s^{1\over\beta})\,s^{\alpha -1} \le {\eta(s)\over s},
				\quad s\in(0,1),\cr
			\sup_{t<s<1}\varphi(s^{1\over\beta})\,s^{\alpha -1} \le \sup_{t < s < 1} {\eta(s)\over s},
				\quad t\in(0,1).
		}$$
		The right hand side of the last inequality equals $\eta(t)/t$ by quasiconcavity of $\eta$.
		Then multiplying by $t$	gives that $\psi(t)\le \eta(t)$ for $t\in(0,1)$.

		Now by Proposition~\ref{prop:opt_ri_domain} we have
		\begin{align*}
			\varphi_{X}(t)
				&= \sup_{h\sim \chi_{(0,t)}}\;
					\biggl\| \int_{s^\beta}^1 y^{\alpha-1}\,h(y)\,\d y \biggr\|_{M(0,1)}
				\ge	\biggl\| \int_{s^\beta}^1 y^{\alpha-1}\,\chi_{(0,t)}(y)\,\d y \biggr\|_{M(0,1)}
					\\
				& = \biggl\| \chi_{(0,t^{1/\beta})}(s) \int_{s^\beta}^t y^{\alpha-1}\,\d y \biggr\|_{M(0,1)}
				\simeq	\Bigl\| \chi_{(0,t^{1/\beta})}(s) \bigl( t^\alpha - s^{\alpha\beta} \bigr) \Bigr\|_{M(0,1)}
					\\
				& \ge	\Bigl\| \chi_{(0,t^{1/\beta}/2)}(s) \bigl( t^\alpha - t^{\alpha}2^{-\alpha\beta} \bigr) \Bigr\|_{M(0,1)}
				\simeq t^\alpha\, \bigl\| \chi_{(0,t^{1/\beta}/2)}(s) \bigr\|_{M(0,1)}
					\\
				& \simeq t^\alpha\, \bigl\| \chi_{(0,t^{1/\beta})}(s) \bigr\|_{M(0,1)}
				= t^\alpha\,\varphi(t^{1\over\beta}).
		\end{align*}
		Hence $\varphi_X(t)\gtrsim t^\alpha\,\varphi(t^{1/\beta})$ and by the claim
		$\psi(t)\lesssim \varphi_{X}(t)$.

		Let us focus on the inequality ``$\lesssim$''. Let $t\in(0,1/2)$ then
		\begin{align*}
			\varphi_X(t)
				& = \sup_{h\sim \chi_{(0,t)}}
					\biggl\| \int_{s^\beta}^1 y^{\alpha-1}\,h(y)\,\d y \biggr\|_{M(0,1)}
					\\
				& = \sup_{h\sim \chi_{(0,t)}} \sup_{0<s<1}
					\biggl( \int_{r^\beta}^1 y^{\alpha-1}\,h(y)\,\d y \biggr)^{\!\!**}\!(s)\,\varphi(s)
					\\
				& = \sup_{h\sim \chi_{(0,t)}} \sup_{0<s<1}
					{\varphi(s)\over s} \int_0^s \int_{r^\beta}^1 y^{\alpha-1}\,h(y)\,\d y\,\d r
					\\
				& = \sup_{h\sim \chi_{(0,t)}} \sup_{0<s<1}
					{\varphi(s)\over s} \int_0^1 y^{\alpha-1}\,h(y) \int_0^{\min\{{y^{1\over\beta}},s\}}\,\d r\,\d y
					\\
				& = \sup_{0<s<1} \sup_{h\sim \chi_{(0,t)}} 
					{\varphi(s)\over s} \biggl( \int_0^{s^\beta} y^{\alpha+{1\over\beta}-1}\,h(y)\,\d y
					+ s \int_{s^\beta}^1 y^{\alpha-1}\,h(y)\,\d y \biggr)
					\\
				& = \sup_{0<s<1} \sup_{\scriptstyle 0<z<s^\beta \atop \scriptstyle s^\beta < z+t <1} 
					{\varphi(s)\over s} \biggl( \int_z^{s^\beta} y^{\alpha+{1\over\beta}-1}\,\d y
					+ s \int_{s^\beta}^{z+t} y^{\alpha-1}\,\d y \biggr).
		\end{align*}
		Denote
		$$
			V(s,z,t)=
				{\varphi(s)\over s} \biggl( \int_z^{s^\beta} y^{\alpha+{1\over\beta}-1}\,\d y
				+ s \int_{s^\beta}^{z+t} y^{\alpha-1}\,\d y \biggr).
		$$
		We split the area over which the supremum is taken into three disjoint regions, namely
		$$
				\varphi_X(t)
					\le	\sup_{\scriptstyle 0<s<t^{1\over\beta } \atop\scriptstyle 0<z<s^\beta} V(s,z,t)
					+ \sup_{\scriptstyle t^{1\over\beta} < s < (1-t)^{1\over\beta} \atop \scriptstyle s^\beta -t<z<s^\beta} V(s,z,t)
					+ \sup_{\scriptstyle (1-t)^{1\over\beta} <s<1 \atop\scriptstyle s^\beta -t<z<1-t} V(s,z,t).
		$$
		Now
		\begin{align*}
			\sup_{\scriptstyle 0<s<t^{1\over\beta } \atop\scriptstyle 0<z<s^\beta} V(s,z,t)
				&\le \sup_{0<s<t^{1\over\beta}} {\varphi(s)\over s} \biggl( \int_0^{s^\beta} y^{\alpha+{1\over\beta}-1}\,\d y
					+ s \int_{s^\beta}^{s^\beta+t} y^{\alpha-1}\,\d y \biggr)
					\\
				&\le \sup_{0<s<t^{1\over\beta}} {\varphi(s)\over s} \biggl( \int_0^{s^\beta} s^{\beta(\alpha-1)+1}\,\d y
					+ s \int_0^{t} y^{\alpha-1}\,\d y \biggr)
					\\
				&\lesssim \sup_{0<s<t^{1\over\alpha}} {\varphi(s)\over s}
					\left( ss^{\alpha\beta} + s t^{\alpha} \right)
					\\
				&\lesssim \varphi(t^{1\over\beta})\,t^{\alpha}
					\\
				&\le t\sup_{t<s<1}\varphi(s^{1\over\beta})\,s^{\alpha-1},
		\end{align*}
		\begin{align*}
			\sup_{\scriptstyle t^{1\over\beta} < s < (1-t)^{1\over\beta} \atop \scriptstyle s^\beta -t<z<s^\beta} V(s,z,t)
				&\le \sup_{t^{1\over\beta}<s<(1-t)^{1\over\beta}} {\varphi(s)\over s} \biggl( \int_{s^\beta-t}^{s^\beta} y^{\alpha+{1\over\beta}-1}\,\d y
					+ s \int_{s^\beta}^{s^\beta+t} y^{\alpha-1}\,\d y \biggr)
					\\
				&\le \sup_{t^{1\over\beta}<s<(1-t)^{1\over\beta}} {\varphi(s)\over s}
					\left( ts^{\beta(\alpha-1)+1} + s t s^{\beta(\alpha-1)} \right)
					\\
				&\lesssim t \sup_{t^{1\over\beta}<s<1}\varphi(s)\,s^{\beta(\alpha-1)}
					\\
				&\lesssim t \sup_{t<s<1}\varphi(s^{1\over\beta})\,s^{\alpha-1}
		\end{align*}
		and
		\begin{align*}
			\sup_{\scriptstyle (1-t)^{1\over\beta} <s<1 \atop\scriptstyle s^\beta -t<z<1-t} V(s,z,t)
				&\le \sup_{(1-t)^{1\over\beta}<s<1} {\varphi(s)\over s} \biggl( \int_{s^\beta-t}^{s^\beta} y^{\alpha+{1\over\beta}-1}\,\d y
					+ s \int_{s^\beta}^{1} y^{\alpha-1}\,\d y \biggr)
					\\
				&\le \sup_{(1-t)^{1\over\beta}<s<1} {\varphi(s)\over s}
					\left( ts^{\beta(\alpha-1)+1} + s(1-s^\beta)s^{\beta(\alpha-1)} \right)
					\\
				&\le \sup_{(1-t)^{1\over\beta}<s<1} {\varphi(s)\over s}
					\left( ts^{\beta(\alpha-1)+1} + ts^{\beta(\alpha-1)+1} \right)
					\\
				&\lesssim t\sup_{(1-t)^{1\over\beta}<s<1}\varphi(s)\,s^{\beta(\alpha-1)}
					\\
				&\lesssim t\sup_{t<s<1}\varphi(s^{1\over\beta})\,s^{\alpha-1}.
		\end{align*}
		Finally
		$$
			\varphi_X(t)\lesssim
				t\sup_{t<s<1}\varphi(s)\,s^{\alpha-1},\quad
				t\in(0,1/2).
		$$
		\end{proof}

	\begin{proof}[Proof of Theorem~B]
		Consider the Orlicz space $L^A(\domain)$ and the Marcinkiewicz space $M(\domain)$
		from the assumption of the theorem. We will prove both implications at once
		using only equivalent steps.
		The statement
		$H_\alpha^\beta\colon L^A(\domain) \to M(\domain)$
		means
		$$
			\biggl\| \int_{t^\beta}^1 g(s)\,s^{{\alpha}-1}\,\d s \biggr\|_{M(\domain)}
				\lesssim \|g\|_{L^A(\domain)},
					\quad g\in\mathcal{M}^+(\domain).
		$$
		Passing to the associate spaces, this is by \citep[Lemma 8.1]{CPS} the same as
		$$
			\biggl\| t^{{\alpha}-1}\int_0^{t^{1\over\beta}} f(s)\,\d s \biggr\|_{L^{\widetilde{A}}(\domain)}
				\lesssim \|f\|_{M'(\domain)},
					\quad f\in \mathcal{M}^+(\domain).
		$$
		where $\widetilde{A}$ is the complementary Young function to $A$.
		This is equivalent to
		$$
			\biggl\| t^{{\alpha}-1}\int_0^{t^{1\over\beta}} f^*(s)\,\d s \biggr\|_{L^{\widetilde{A}}(\domain)}
				\lesssim \|f\|_{M'(\domain)},
					\quad f\in M'(\domain).
		$$
		Indeed, one implication follows just by passing to only nonincreasing functions with
		the fact that $\|f\|_{M'(\domain)}=\|f^*\|_{M'(\domain)}$, and the other holds thanks to the Hardy-Littlewood
		inequality applied to functions $f$ and $\chi_{(0,t^{1/\beta})}$.

		Using the fact that $M'(\domain)$ is a Lorentz endpoint space and passing to the
		characteristic functions while keeping
		Proposition~\ref{prop:Lambda_space_embedding} in mind, this is equivalent to
		\begin{equation}
			\biggl\| t^{{\alpha}-1} \int_0^{t^{1\over\beta}} \chi_{(0,a)}(s)\,\d s \biggr\|_{L^{\widetilde{A}}(\domain)}
				\lesssim \varphi_{M'}(a),\quad a\in(0,1).
				\label{eq:char_in_op_ineq}
		\end{equation}

		Let us compute the left hand side. Clearly
		\begin{align*}
			\biggl\| t^{{\alpha}-1} \int_0^{t^{1\over\beta}} \chi_{(0,a)}(s)\,\d s \biggr\|_{L^{\widetilde{A}}(\domain)}
				&= \| t^{{\alpha}-1} \chi_{(0,a^\beta)}(t)\cdot t^{1\over\beta} 
					+ t^{{\alpha}-1} \chi_{(a^\beta,1)}(t)\cdot a \|_{L^{\widetilde{A}}(\domain)}
					\\
				&\le \| t^{{\alpha}+{1\over\beta}-1} \chi_{(0,a^\beta)}(t) \|_{L^{\widetilde{A}}(\domain)}
					+ a\, \|t^{{\alpha}-1} \chi_{(a^\beta,1)}(t) \|_{L^{\widetilde{A}}(\domain)}.
		\end{align*}
		We suppose that $a\in(0,2^{-{1/\beta}})$, since we are interested only in
		values of $a$ near zero.
		We show that the second summand dominates the first one.
		Indeed, for any r.i.\ norm we have
		$$\displaylines{\quad
			a \|t^{{\alpha}-1} \chi_{(a^\beta,1)}(t) \| 
				\ge a \|t^{{\alpha}-1} \chi_{(a^\beta,2a^\beta)}(t) \| 
				\ge a(2a^\beta)^{{\alpha}-1}\| \chi_{(a^\beta,2a^\beta)}(t) \| 
				\hfill\cr\hfill
				\simeq a^{\beta(\alpha-1)+1} \| \chi_{(0,a^\beta)}(t) \| 
				\ge \| t^{\alpha+{1\over\beta}-1} \chi_{(0,a^\beta)}(t) \| 
		\quad}$$
		Therefore we can state that
		\begin{equation}
			\biggl\| t^{{\alpha}-1} \int_0^{t^{1\over\beta}} \chi_{(0,a)}(s)\,\d s \biggr\|_{L^{\widetilde{A}}(\domain)}
				\simeq  a\, \|t^{{\alpha}-1} \chi_{(a^\beta,1)}(t) \|_{L^{\widetilde{A}}(\domain)}.
					\label{eq:chareq}
		\end{equation}
		At this moment, it is the time for using Lemma~\ref{lemm:function_E_properties_i}.
		We need the part \eqref{eq:Function_Eiii}
		with \eqref{eq:Function_Eii} for $\xi={\alpha-1}<0$, $R=1$ and
		$\varepsilon=1-2^{-{1/\beta}}$. The assumption \eqref{eq:converges_near_zero} can be rendered as satisfied
		without any loss of generality since the domain is of finite
		measure, hence the appropriate Young function can be redefined on $(0,1)$
		without any effect to the corresponding Orlicz space. Note also that we are using the complementary Young
		function $\widetilde{A}$ instead of $A$.
		Hence we conclude that \eqref{eq:char_in_op_ineq} is equivalent to
		\begin{equation}
			{a^{\beta(\alpha-1)+1}\over E_{{\alpha}-1}^{-1}(a^{-\beta})}
				 \lesssim \varphi_{M'}(a),\quad a\in(0,2^{-{1\over\beta}}).
				 \label{eq:phiMvsE}
		\end{equation}

		Now we substitute $t=a^{-\beta}$ and use the fact that $\varphi_{M'}(a)\,\varphi(a)=a$.
		We get
		\begin{equation}
			\varphi(t^{-{1\over\beta}})\,t^{1-{\alpha}}
				\lesssim E_{{\alpha}-1}^{-1}(t),\quad t\in(2^{1\over\beta},\infty).
				\label{eq:phivsE}
		\end{equation}
		Let us 	define
		$$
			F(t)=\bar{\varphi}(t^{-{1\over\beta}})\,t^{1-{\alpha}},
				\quad t\in (0,1),
		$$
		where the function $\bar{\varphi}(t)$ is taken from Lemma~\ref{lemm:quasi}.
		We claim that $F(t)$ is the least nondecreasing majorant of $\varphi(t^{-{1/\beta}})\,t^{1-{\alpha}}$.
		Indeed,
		$$
			\bar{\varphi}(t)=
				t^{\beta(1 - \alpha)}\sup_{t<s<1} \varphi(s)\,s^{\beta(\alpha - 1)},
					\quad t\in (0,1),
		$$
		hence
		$$
				\bar\varphi(t^{-{1\over\beta}})\,t^{1-{\alpha}}
					= \sup_{0<s<t} \varphi(s^{-{1\over\beta}})\,s^{1-{\alpha}},
					\quad t\in(0,1),
		$$
		and the claim follows.

		Since the function $E_{{\alpha}-1}$ is strictly increasing as well as its
		inverse, we can enlarge the left hand side of the inequality \eqref{eq:phivsE}
		by $F(t)$.  Hence we can equivalently continue by
		\begin{equation}
				F(t) \lesssim E_{{\alpha}-1}^{-1}(t),\quad t\in(2^{1\over\beta},\infty).
					\label{eq:FvsE}
		\end{equation}

		Now Lemma~\ref{lemm:YoungB} comes to play with $u(t)=\bar\varphi(t^{1/\beta})\,t^\alpha$.
		By Lemma~\ref{lemm:quasi} $u$ is quasiconcave and strictly increasing on~$(0,1)$.
		Next,
		$$
			\lim_{t\to 0^+} {u(t)\over t}
				= \lim_{t\to 0^+} \sup_{t<s<1} \varphi(s^{1\over\beta})\,s^{\alpha-1}
				= \infty
		$$
		thanks to the assumption of the theorem. Also, $u$ is right continuous at the origin
		since
		$$
			u(t) = t \sup_{t<s<1} \varphi(s^{1\over\beta})\,s^{\alpha-1}
				\le \varphi (1)\, t^\alpha.
		$$

		We obtain a Young function $B$ such
		that $\widetilde{B}^{-1}(t)\simeq F(t)$.
		Theorem~C ensures that the space $L^B(0,1)$
		has the same fundamental function as the optimal r.i.\ domain $X(0,1)$
		in $H_\alpha^\beta\colon X(0,1)\to M(0,1)$.
		Using this and passing to inverse functions,
		\eqref{eq:FvsE} is equivalent to the existence of some constant $C>0$ such that
		$$
				E_{{\alpha}-1}(t) \le \widetilde{B}(Ct),\quad t\in(c,\infty),
		$$
		where $c=E^{-1}_{{\alpha}-1}(2^{1/\beta})>0$. This is however equivalent to
		$$
				E_{{\alpha}-1}(t) \lesssim \widetilde{B}(Ct),\quad t\in(2,\infty),
		$$
		which is nothing but
		$$
			\int_0^t {\widetilde{A}(s)\over s^{{1/(1-\alpha)}+1}}\,\d s
				\lesssim {\widetilde{B}(Ct)\over t^{1/(1-\alpha)}},
					\quad t\in(2,\infty).
		$$

		Finally observe that the quantities
		$\int_0^t \widetilde{A}(s)\, s^{{1/(\alpha-1)}-1}\,\d s$
		and
		$\int_1^t \widetilde{A}(s)\, s^{{1/(\alpha-1)}-1}\,\d s$
		are comparable since $t\in(2,\infty)$.
		One can now immediately observe that the resulting inequality does
		not depend on the behavior of the Young function $\widetilde{A}$
		on the interval $(0,1)$.
		\end{proof}

\section{Proof of Theorem~A} \label{sec:main}

	Before proving Theorem~A we need
	several auxiliary results. The next theorem
	is the crucial ingredient in the proof of the main result
	and it reveals the constructive approach to the nonexistence
	of an optimal Orlicz domain space in appropriate situations.
	
	\begin{theorem}\label{thm:construct}
		Let Young functions $A$ and $B$ satisfy for $0<\alpha<1$
		\let\widetilde\relax
		and some $C>0$ the inequality
		\begin{equation}
			\int_1^t {\widetilde{A}(s)\over s^{{1/(1-\alpha)}+1}}\,\d s
				\lesssim {\widetilde{B}(Ct)\over t^{1/(1-\alpha)}},
					\quad t\in(2,\infty).
				\label{eq:red}
		\end{equation}
		If
		\begin{equation}
			\lim_{t\to\infty} {B(t)\over t^{1/(1-\alpha)}}=\infty
				\label{eq:unbound_G}
		\end{equation}
		and
		\begin{equation}
			\limsup_{t\to\infty} {t^{1/(1-\alpha)}\over B(Kt)} \int_1^t {B(s)\over s^{1/(1-\alpha)-1}}\,\d s=\infty
				\label{eq:limsup_cond}
		\end{equation}
		for every $K\ge 1$, then there exists a Young function $A_1$ satisfying
		$A_1 \esucc A$
		and
		$$
			\int_1^t {\widetilde{A}_1(s)\over s^{{1/(1-\alpha)}+1}}\,\d s
				\lesssim {\widetilde{B}(Ct)\over t^{1/(1-\alpha)}},
					\quad t\in(2,\infty).
		$$
		\end{theorem}

	\begin{proof}
		Let $A$ and $B$ be the Young functions from the assumptions.
		\let\widetilde\relax
		First,
		we establish an upper bound for $\widetilde{A}$. Namely, for $t\in(1,\infty)$
		\begin{equation*}
				{B(2Ct)\over t^{1/(1-\alpha)}}
				\gtrsim \int_1^{2t} {\widetilde{A}(s)\over s^{{1/(1-\alpha)}+1}}\,\d s
				\ge \int_t^{2t} {\widetilde{A}(s)\over s^{{1/(1-\alpha)}+1}}\,\d s
				\ge \widetilde{A}(t)\int_t^{2t} {1\over s^{{1/(1-\alpha)}+1}}\,\d s
				\simeq {\widetilde{A}(t)\over t^{1/(1-\alpha)}}.
		\end{equation*}
		Using this, we obtain the existence of $\gamma>0$ such that
		\begin{equation}
			\gamma\, \widetilde{B}(2Ct)>\widetilde{A}(t),\quad t\in (1,\infty).
				\label{eq:boundA}
		\end{equation}
		Now we fix this $\gamma$ and, for every $t\in(1,\infty)$, we define the set
		$$
			G_t=\bigl\{s\in(1,\infty),\;{\textstyle{ \widetilde{A}(s)\over s}}
				\ge \gamma\,{\textstyle{\widetilde{B}(2Ct)\over t}}\bigr\}.
		$$
		Since $\widetilde{A}(s)/s$ is a nondecreasing mapping from $(0,\infty)$ onto itself, the sets
		are nonempty for every $t\in(1,\infty)$.
		Let us define $\tau=\tau_t=\inf G_t$.
		Observe that, for $t\in(1,\infty)$ and $s\in(1,t)$,
		$$
			{\widetilde{A}(s)\over s} \le {\widetilde{A}(t)\over t} < \gamma\,{\widetilde{B}(2Ct)\over t}
		$$
		thanks to the estimate \eqref{eq:boundA}.
		Hence $\tau_t > t$ for every $t$. Moreover, since
		$\widetilde{A}(t)/t$ is continuous, we have the equality
		\begin{equation}
			{\widetilde{A}(\tau)\over \tau}=\gamma\, {\widetilde{B}(2Ct)\over t},\quad t\in (1,\infty).
				\label{eq:tauteq}
		\end{equation}

		Let $K$ be a real number such that $K\ge 1$. Then
		\begin{equation}
			\limsup_{t\to\infty}{\widetilde{A}(\tau_t)\over\tau_t}\cdot {t\over \widetilde{A}(2Kt)}=\infty.
				\label{eq:unbound_for_all_M}
		\end{equation}
		Indeed, suppose that there exist $K\ge 1$ and some $L>0$ such that there
		is for all $t\in(1,\infty)$ the estimate
		$$
			{\widetilde{A}(\tau)\over\tau}\cdot {t\over \widetilde{A}(2Kt)}<L,
		$$
		or equivalently
		\begin{equation}
			{\widetilde{A}(2Kt)\over t}>L^{-1}\,{\widetilde{A}(\tau)\over\tau}.
				\label{eq:tmpbound}
		\end{equation}
		Now for $t>2$ the following holds:
		\begin{align*}
			{B(CKt)\over t^{1/(1-\alpha)}}
				&\gtrsim \int_1^{Kt} {\widetilde{A}(s)\over s^{{1/(1-\alpha)}+1}}\,\d s 
				\ge \int_K^{Kt} {\widetilde{A}(s)\over s^{{1/(1-\alpha)}+1}}\,\d s
				 	\\
				&\simeq \int_{1/2}^{t/2} {\widetilde{A}(2Ks)\over s^{{1/(1-\alpha)}+1}}\,\d s 
					\tag{by change of variables} \\
				&\gtrsim \int_{1/2}^{t/2}{\widetilde{A}(\tau_s)\over \tau_s}{1\over s^{{1/(1-\alpha)}}}\,\d s
					\tag{by \eqref{eq:tmpbound}} \\
				&\simeq  \int_{1/2}^{t/2}{B(2Cs)\over s}{1\over s^{1/(1-\alpha)}}\,\d s
					\tag{by \eqref{eq:tauteq}}\\
				&\simeq 	\int_1^{t}{B(Cs)\over s^{1/(1-\alpha)+1}}\,\d s.
		\end{align*}
		After the change of variables $Cs\mapsto s$, this contradicts \eqref{eq:limsup_cond} for this $K$.

		From estimate \eqref{eq:unbound_for_all_M}, we can take an increasing sequence
		$t_j\in(2,\infty)$, $j\ge 2$, such that
		\begin{equation}
			\lim_{j\to\infty}{\widetilde{A}(\tau_j)\over\tau_j}\cdot{t_j\over \widetilde{A}(jt_j)}=\infty,
				\label{eq:subseq}
		\end{equation}
		where we define $\tau_j=\tau_{t_j}$. We can also choose this sequence to ensure $t_{j+1}>\tau_j$.
		We claim that without loss of
		generality we can assume that $2t_j<\tau_j$ for every index $j\ge 2$.
		Indeed, suppose that there exists a subsequence ${j_k}$ in $\N$ such that
		$\tau_{j_k}\le 2t_{j_k}$. Then $\widetilde{A}(\tau_{j_k})\le
		\widetilde{A}(2t_{j_k})$ and
		$$
			{\widetilde{A}(\tau_{j_k})\over\tau_{j_k}}\cdot {t_{j_k}\over \widetilde{A}(j_k t_{j_k})}
				\le {\widetilde{A}(2t_{j_k})\over t_{j_k}}\cdot {t_{j_k}\over \widetilde{A}({j_k\over 2} 2t_{j_k})}
				\le {\widetilde{A}(2t_{j_k})\over \widetilde{A}(2t_{j_k})}\cdot {2\over j_k}
				= {2\over j_k}\to 0
					\quad\text{as}\quad k\to\infty,
		$$
		which is impossible due to \eqref{eq:subseq}.

		At this moment, we can define a function $\widetilde{A}_1$ by the formula
		$$
			\widetilde{A}_1(t) =
			\begin{cases}
				\widetilde{A}(t_j) +
				{ \widetilde{A}(\tau_j) - \widetilde{A}(t_j) \over \tau_j - t_j}\,(t - t_j),
					& t\in(t_j,\tau_j),\;j\in\N,\\
				\widetilde{A}(t),
					& \text{otherwise}.
			\end{cases}
		$$
		Obviously, $\widetilde{A}_1\ge \widetilde{A}$ and $\widetilde{A}_1$ is a Young function.
		Moreover, for $j\in\N$, $j\ge 2$,
		\begin{align*}
			{\widetilde{A}_1(2t_j)\over \widetilde{A}(jt_j)}
				& ={\widetilde{A}(t_j) +
					{ \widetilde{A}(\tau_j) - \widetilde{A}(t_j) \over \tau_j - t_j}
						\, t_j \over \widetilde{A}(jt_j)}
					\\
				& \ge { \widetilde{A}(\tau_j) - \widetilde{A}(t_j) \over \widetilde{A}(jt_j)}
					\cdot {t_j \over \tau_j}
					\\
				& \ge { \widetilde{A}(\tau_j) - \widetilde{A}({\textstyle{\tau_j\over 2}})
						\over \widetilde{A}(jt_j)}
					\cdot {t_j \over \tau_j}
					\tag{since $2t_j < \tau_j$}
					\\
				& \ge {1\over 2}\cdot {\widetilde{A}(\tau_j)\over \tau_j}
					\cdot {t_j\over \widetilde{A}(jt_j)}
					\tag{since $\widetilde{A}(\tau_j/2)\le \widetilde{A}(\tau_j)/2$},
		\end{align*}
		and the latter tends to infinity as $j\to\infty$ by \eqref{eq:subseq}.
		Therefore
		$$
			\limsup_{t\to\infty}{\widetilde{A}_1(t)\over \widetilde{A}(\lambda t)}=\infty
		$$
		for every $\lambda>2$, which is precisely $\widetilde{A}_1\esucc \widetilde{A}$.

		It remains to show that $\widetilde{A}_1$ satisfies the condition \eqref{eq:red}
		with $A$ replaced by $A_1$.
		Let $t\in(2,\infty)$ be fixed. We find $j\in\N$ such shat $t\in[t_j,t_{j+1})$.
		Then we have
		\begin{align*}
			\int_1^t {\widetilde{A}_1(s)\over s^{{1/(1-\alpha)}+1}}\,\d s
				&\le \int_1^t {\widetilde{A}(s)\over s^{{1/(1-\alpha)}+1}}\,\d s
					+ \sum_{k=1}^j \int_{t_k}^{\tau_k}\biggl(\widetilde{A}(t_k)
					+ {\widetilde{A}(\tau_k) - \widetilde{A}(t_k) \over \tau_k - t_k}\,(s-t_k)\biggl){1\over s^{{1/(1-\alpha)}+1}}\,\d s
					\\
				&\le 2\int_1^t {\widetilde{A}(s)\over s^{{1/(1-\alpha)}+1}}\,\d s
					+ \sum_{k=1}^j {\widetilde{A}(\tau_k) - \widetilde{A}(t_k) \over \tau_k - t_k}
						\int_{t_k}^{\tau_k} {(s-t_k)\over s^{{1/(1-\alpha)}+1}}\,\d s.
		\end{align*}
		We can follow with estimates of the latter integral. Since ${1/(1-\alpha)}>1$, we have
		for $k\in\N$ such that $1\le k\le j$,
		$$
			\int_{t_k}^{\tau_k}{(s-t_k)\over s^{{1/(1-\alpha)}+1}}\,\d s
				\le \int_{t_k}^{\tau_k} {1\over s^{{1/(1-\alpha)}}}\,\d s
				\le \int_{t_k}^{\infty} {1\over s^{{1/(1-\alpha)}}}\,\d s
				\simeq {1\over t_k^{{1/(1-\alpha)}-1}}.
		$$
		This together with the fact that $2t_k<\tau_k$ gives
		$$
			\int_1^t {\widetilde{A}_1(s)\over s^{{1/(1-\alpha)}+1}}\,\d s
				\lesssim 2\int_1^t {\widetilde{A}(s)\over s^{{1/(1-\alpha)}+1}}\,\d s
					+ 2\sum_{k=1}^j {\widetilde{A}(\tau_k)\over \tau_k} {1\over t_k^{{1/(1-\alpha)}-1}}.
		$$
		Since \eqref{eq:tauteq} implies
		$$
			{\widetilde{A}(\tau_k)\over \tau_k} {1\over t_k^{{1/(1-\alpha)}-1}}
				= \gamma\,{B(2Ct_k)\over {t_k}^{1/(1-\alpha)}}
		$$
		we have by \eqref{eq:red}
		$$
			\int_1^t {\widetilde{A}_1(s)\over s^{{1/(1-\alpha)}+1}}\,\d s
				\lesssim {B(Ct)\over t^{1/(1-\alpha)}}
					+ \sum_{k=1}^j {B(2Ct_k)\over {t_k}^{1/(1-\alpha)}}.
		$$
		Because the sequence $t_j$ could be taken arbitrarily fast growing,
		we can assume without loss of generality that
		$$
			{B(Ct)\over {t}^{1/(1-\alpha)}}
				\ge \sum_{k=1}^{i-1} {B(2Ct_k)\over {t_k}^{1/(1-\alpha)}},
					\quad t\in(t_i,\infty),
		$$
		thanks to the assumption \eqref{eq:unbound_G}.
		Adding all the estimates together, we finally obtain that
		$$
			\int_1^t {\widetilde{A}_1(s)\over s^{{1/(1-\alpha)}+1}}\,\d s
				\lesssim {B(Ct)\over t^{1/(1-\alpha)}}
		$$
		which proves the theorem.
		\end{proof}

	The following auxiliary fact is based on the idea of L'H\^ opital's rule and the
	proof is very similar to the proof of the original result, hence we omit it.

	\begin{proposition} \label{prop:Lhop}
		Suppose that $f$ and $g$ are real functions
		having finite derivatives on some neighborhood of infinity.
		If $g(x)\to\infty$ as $x\to\infty$, then
		$$
			\liminf_{x\to\infty} {f'(x)\over g'(x)} \le \liminf_{x\to\infty} {f(x)\over g(x)}.
		$$
		\end{proposition}

	\begin{theorem} \label{thm:eqcondition}
		Let $G\colon(0,\infty)\to(0,\infty)$ be a continuous nondecreasing function satisfying $\Delta_2$ condition. Then the
		following are equivalent.
		\begin{enumerate}[(i)]
			\item $$\limsup_{t\to\infty} {1\over G(Kt)}\int_1^t {G(s)\over s}\,\d s = \infty\quad\text{for every $K\ge 1$};$$
			\item $$\limsup_{t\to\infty} {1\over G(t)}\int_1^t {G(s)\over s}\,\d s = \infty;$$
			\item $$\liminf_{t\to\infty} {G(Kt)\over G(t)}=1\quad\text{for every $K\ge 1$}.$$
		\end{enumerate}
		\end{theorem}

	\begin{proof}
		The equivalence (ii)$\Leftrightarrow$(i) is trivial,
		since the quantities $G(t)$ and $G(Kt)$ are comparable for every fixed $K\ge 1$ 
		thanks to the fact that $G\in\Delta_2$.

		Let us focus on the implication	(iii)$\Rightarrow$(ii). Let $K\ge 1$ be fixed
		and suppose $t>1$. Then
		$$
			\int_1^{Kt} G(s){\d s\over s}
				\ge \int_t^{Kt} G(s){\d s\over s}
				\ge G(t) \int_t^{Kt} {\d s\over s}
				= G(t)\log K.
		$$
		Dividing both sides by $G(Kt)$ we obtain
		$$
			\log K{G(t)\over G(Kt)}
				\le {1\over G(Kt)} \int_1^{Kt} {G(s)\over s}\,\d s.
		$$
		Taking the limes superior as $t\to\infty$ on both sides of the inequality, we get
		$$
			\log K = \log K\, \limsup_{t\to\infty} {G(t)\over G(Kt)}
				\le \limsup_{t\to\infty} {1\over G(Kt)}\int_1^{Kt} {G(s)\over s}\,\d s
				=:L,
		$$
		where $L$ is independent of $K$. Since $\log K\le L$ for arbitrary $K$, $L$
		has no other option but to equal infinity.
		
		To prove (ii)$\Rightarrow$(iii), let $K\ge 1$ be fixed and let us define
		$f(t)=\int_1^t G(Ks){\d s\over s}$ and $g(t)=\int_1^t G(s){\d s\over s}$.
		Then both $f$ and $g$ are continuous and have derivatives, namely
		$f'(t)=G(Kt)/t$, $g'(t)=G(t)/t$. Since (ii) holds, it has to be
		$g(t)\to\infty$ as $t\to\infty$.  Using Proposition~\ref{prop:Lhop},
		we get
		\begin{align*}
				0&\le\liminf_{t\to\infty} {G(Kt) \over G(t)} - 1
					\\
				&\le\liminf_{t\to\infty} {\int_1^t G(Ks){\d s\over s} \over \int_1^t G(s){\d s\over s}} - 1
					\\
				&\le\liminf_{t\to\infty} {\int_K^{Kt} G(s){\d s\over s} - \int_1^t G(s){\d s\over s}\over \int_1^t G(s){\d s\over s}}
					\\
				&\le\liminf_{t\to\infty} {G(t) \over \int_1^t G(s){\d s\over s}}{\int_t^{Kt} G(s){\d s\over s}\over G(t)}.
		\end{align*}
		Since $\liminf_{t\to\infty} {G(t) / \int_1^t G(s){\d s\over s}}=0$, it
		suffices to show that ${1 \over G(t)}\int_t^{Kt} G(s){\d s\over s}$ is
		bounded. To this end we use the fact that $G$ is nondecreasing and, due to
		$G\in\Delta_2$, there is some $c>0$ such that $G(Kt)\le c\,G(t)$ for big
		$t$. For such a $t$ we have
		$$
			{1 \over G(t)} \int_t^{Kt} G(s){\d s\over s}
			\le {G(Kt) \over G(t)} \int_t^{Kt} {\d s\over s}
			\le c \log K.
		$$
		\end{proof}

	\begin{proof}[Proof of Theorem~A]
		The equivalence of (ii) and (v) follows directly from Theorem~B.

		The condition (v) holds if and only if (iv) holds thanks to the consequence of \citep[Theorem A]{KPP}.

		In order to show (i)$\Rightarrow$(v) assume that (v) is not satisfied,\ i.e.,
		$$
			\limsup_{t\to\infty}{1\over G(Kt)}\int_1^t{G(s)\over s}\,\d s=\infty
		$$
		for some constant $K>1$,
		where $G(t)=\widetilde{B}(t)\,t^{1/(\alpha -1)}$.
		Now for any Orlicz space $L^A(0,1)$ satisfying
		$H_\alpha^\beta\colon L^A(0,1)\to M(0,1)$ there exists a constant $C_A$ such that
		$$
			\int_1^t {\widetilde{A}(s)\over s^{{1/(1-\alpha)}+1}}\,\d s
				\lesssim {\widetilde{B}(C_{\!A}t)\over t^{1/(1-\alpha)}},
					\quad t\in(2,\infty)
		$$
		due to Theorem~B. If the function $G$ is unbounded,
		then Theorem~\ref{thm:construct} ensures the existence of a Young function
		$A_1$ such that the space $L^{A_1}(0,1)$ is strictly larger than $L^{A}(0,1)$
		and still renders the inequality above true, with possibly different constants.
		Now again by Theorem~B one has
		$H_\alpha^\beta\colon L^{A_1}(0,1)\to M(0,1)$ and no optimal Orlicz domain exists.
		This contradicts (i).
		
		The case when the function $G$ is bounded and hence equivalent to a constant function,
		corresponds to the situation when $M(0,1)=L^{\infty}(0,1)$. Then no optimal Orlicz
		domain exists thanks to a different construction described in \citep[Theorem 6.4]{CP}.
		This also contradicts (i).

		To prove (iii)$\Rightarrow$(i) we claim that $L^B(0,1)$ is among the Orlicz spaces $L^A(0,1)$
		the largest space rendering
		$$
			H_\alpha^\beta\colon L^A(0,1)\to M(0,1).
		$$
		Indeed let $L^A(0,1)$ be any of such spaces. By the optimality of $X(0,1)$,
		we have $L^A(0,1)\subseteq X(0,1)$ and thus we have the inequality between
		appropriate fundamental functions
		$$
			\varphi_{X}(t)\lesssim \varphi_{L^{A}}(t),\quad t\in (0,1).
		$$
		Since the space $L^B(0,1)$ is defined in a way that its fundamental function coincides
		with $\varphi_X$, one gets that $\varphi_{L^B}(t)\lesssim \varphi_{L^{A}}(t)$
		which implies $A(t)\le B(Ct)$ for some $C>0$, hence
		$L^A(0,1)\subseteq L^B(0,1)$ and $L^B(0,1)$ is optimal.
		
		The equivalence of (ii) and (iii) follows directly from the definition of the
		optimal r.i.\ space, and the equivalence of (v) and (vi) has already been proved in
		Theorem~\ref{thm:eqcondition}.
		\end{proof}

	\begin{remark} Note that the proof of the implication (iii)$\Rightarrow$(i) does  not depend
		on the target space, so it can be used to prove the optimality in positive
		cases for any r.i.\ target space $Y$.
		\end{remark}

\section{Examples and applications} \label{sec:app}

\subsection{Sobolev embeddings on John domains}

We begin by the easiest case of Sobolev embeddings, namely those acting on John domains.
We will use the reduction theorem from \cite{CPS}.
Recall that a bounded open 
set $\Omega$ in $\R^n$ is called a John domain if there exist a constant $c\in (0,1)$ and a point $x_0\in\Omega$
such that for every $x\in\Omega$ there exists a rectifiable curve $\varpi\colon [0,l]\to\Omega$, parameterized 
by arclength, such that $\varpi(0)=x$, $\varphi(l)=x_0$, and
$$
	\dist\bigl(\varpi(r),\partial\Omega\bigr) \le cr,\quad r\in[0,l].
$$

We will use the reduction principle for John domains 
proved in \citep[Theorem 6.1]{CPS}. It can be read as follows.

		Let $n\in\N$, $n\geq 2$, and let $m\in\N$.
		Assume that $\Omega$ is a John domain in $\R^n$. Let $\| \cdot
		\|_{X(0,1)}$ and $\| \cdot \|_{Y(0,1)}$ be rearrangement-invariant
		function norms. Then the following assertions are equivalent.

		\begin{enumerate}[(i)]
		\item The Hardy type inequality
		$$
			\bigl\| H_{m/n}^1 f \bigr\|_{ Y(0,1)}
				\leq C \|f\|_{X(0,1)}
		$$
		holds for some constant $C$ and for every nonnegative $f\in
		X(0,1)$.

		\item The Sobolev embedding
		$$
			W^mX(\Omega) \embed Y(\Omega)
		$$
		holds.
		\end{enumerate}

Recall that
$$
	W^mX(\Omega) =
		\Bigl\{u\in\mathcal{M}(\Omega),\; u\text{ is $m$-times weakly differentiable in $\Omega$ and }
			\bigl|\nabla^k u\bigr|\in X(\Omega),\, k=0,1,\ldots, m\Bigr\}.
$$
Here, $\nabla^k u$ denotes the vector of all $k$-th order weak derivatives of
$u$ and $\nabla^0 u = u$.  The norm is then given by
$$
	\|u\|_{W^mX(\Omega)} :=
		\sum_{k=0}^{m} \,\bigl\|\nabla^k u \bigr\|_{X(\Omega)}.
$$

Now, one can select any r.i.\ space $X(0,1)$ and seek to find an optimal range
space. Let $\Omega$ be a John domain in~$\R^n$, $m\in\N$ such that $m<n$ and
consider the spaces $L^p\log^qL(\Omega)$ or $L^p\log^q\log L(\Omega)$, $p>1$
and $q\in\R$ or $p=1$ and $q\ge 0$. 
By \citep[Theorem~6.12 and Example~6.14]{CPS} (see also \citep[Examples~1 and~2]{AC}), we have
$$
	W^mL^p\log^qL(\Omega) \embed
	\begin{cases}
		L^{np\over n-mp}\log^{nq\over n-mp}L(\Omega),
			& 1\le p < {n\over m}, \\
		\exp L^{n\over n-m(1+q)}(\Omega),
			& p={n\over m}, q<{n\over m} -1, \\
		\exp\exp L^{n\over n-m}(\Omega),
			& p={n\over m}, q={n\over m} -1, \\
		L^\infty(\Omega),
			& p>{n\over m} \text{ or } p={n\over m}, q>{n\over m}-1,
	\end{cases}
$$
and
$$
	W^mL^p\log^q\log L(\Omega) \embed
	\begin{cases}
		L^{np\over n-mp}\log^{nq\over n-mp}\log L(\Omega),
			& 1\le p < {n\over m}, \\
		\exp \bigl(L^{n\over n-m}\log^{mq\over n-m}L\bigr)(\Omega),
			& p={n\over m}, \\
		L^\infty(\Omega),
			& p>{n\over m},
	\end{cases}
$$
and all the targets are optimal among all Orlicz spaces.

Let us investigate the optimal Orlicz domains.

	\begin{example} \label{ex:applications}
	\noindent a) Case $Y(\Omega)=L^{np\over n-mp}\log^{nq\over n-mp}L(\Omega)$, $1\le p<{n\over m}$.
		The space $Y(\Omega)$ is not a
		Marcinkiewicz space, but instead of $Y(\Omega)$ we can take the endpoint space $M_\varphi(\Omega)$
		with the same fundamental function as the space $Y(\Omega)$, namely
		$$
			\varphi(t)=t^{n-mp\over np}\log^{q\over p}\bigl(\textstyle {2\over t}\bigr),
				\quad t\in(0,1).
		$$
		Now, thanks to reduction principle the problem of Sobolev embedding
		is equivalent to the boundedness of the operator~$H^1_{m/n}$.

		By Theorem~A, the optimal Orlicz domain space exists
		if and only if $H^1_{m/n}\colon L^B(0,1)\to M_\varphi(0,1)$ where, after some calculations,
		$B(t)=t^p\log^q t$ for large $t$. This is however the same as
		$W^mL^p\log^qL(\Omega)\embed M_\varphi(\Omega)$ which
		is satisfied since
		$$
			W^mL^p\log^qL(\Omega)\embed Y(\Omega)\subseteq M_\varphi(\Omega).
		$$
		Hence both domain and range spaces in
		$$
			W^mL^p\log^qL(\Omega)\embed L^{np\over n-mp}\log^{nq\over n-mp}L(\Omega)
		$$
		are optimal among Orlicz spaces.

	\noindent b) Case $Y(\Omega)=\exp L^{n\over n-m(1+q)}(\Omega)$, $p={n\over m}$, $q<{n\over m}-1$.
		The Orlicz space $Y(\Omega)$ coincides with the Marcinkiewicz endpoint
		space $M_\varphi(\Omega)$ (cf.\ \eqref{eq:orlismarc}) where
		$$
			\varphi(t)=\log^{{m\over n}(1+q)-1}\bigl(\textstyle {2\over t}\bigr),
				\quad t\in(0,1).
		$$
		Again by reduction principle and Theorem~A
		we compute the Young function $B$ and test the boundedness of $H^1_{m/n}$
		on the space $L^B(0,1)$ or check the condition using the function $G(t)=\widetilde{B}(t)\,t^{n/(m-n)}$.
		We get
		$$
			G(t)=\log^{n-m(1+q)\over n-m}(t),
				\quad t\in (1,\infty).
		$$
		Since
		$$
			\liminf_{t\to\infty} {\log^{n-m(1+q)\over n-m}(Ct)\over \log^{n-m(1+q)\over n-m}(t) }=1,
				\quad\text{for every } C\ge 1
		$$
		and $G$ satisfies the $\Delta_2$ condition we conclude that
		the space $L^p\log^q L(\Omega)$ is not the largest Orlicz space
		rendering
		$$
			W^mL^p\log^q L(\Omega)\embed \exp L^{{n\over n-m(1+q)}}(\Omega)
		$$
		and no such Orlicz space exists.
		Just to compare, the space $L^B(0,1)$ from Theorem~A
		is $L^p\log ^{1+q-p} L(0,1)$ which is too large.
		\end{example}

These two examples give us the outline how to use our results to investigate
the optimal Orlicz domains.  Other cases can be done in an analogous way and we
just present the results (see Table~\ref{tab:1}).  Observe that the optimal
Orlicz domains exist in subcritical cases, i.e.\ when $1\le p<{n\over m}$,
otherwise every Orlicz domain space can be improved.

\begin{table*}
\centering
\renewcommand{\arraystretch}{1.4}
	\begin{tabular}{llll}
	$Y(\Omega)$
		& 
		& $L^B(\Omega)$
		& $G(t)$
		\\
	\hline
	$L^{np\over n-mp}\log^{nq\over n-mp}L(\Omega)$
		& $1\le p< {n\over m}$
		& $L^p\log^qL(\Omega)$
		& $t^{{n\over m-n}+{p\over p-1}}\log^{q\over 1-p}(t)$
		\\
	$L^{np\over n-mp}\log^{nq\over n-mp}\log L(\Omega)$
		& $1\le p< {n\over m}$
		& $L^p\log^q\log L(\Omega)$
		& $t^{{n\over m-n}+{p\over p-1}}\log^{q\over 1-p}\log(t)$
		\\
	$\exp L^{n\over n-m(1+q)}(\Omega)$
		& $p = {n\over m}$, $q<{n\over m}-1$
		&	$L^p\log^{1+q-p}L(\Omega)$
		& $\log^{n-m(1+q)\over n-m}(t)$
		\\
	$\exp\exp L^{n\over n-m}(\Omega)$
		& $p = {n\over m}$, $q={n\over m}-1$
		&	$L^p\log^{-q}\log L(\Omega)$
		& $\log\log (t)$
		\\
	$\exp \bigl(L^{n\over n-m}\log^{mq\over n-m}L\bigr)(\Omega)$
		& $p = {n\over m}$
		&	$L^p\log^{1-p} L\log^{q}\log L(\Omega)$
		& $\log(t)\log^{mq\over m-n}\log(t)$
		\\
	$L^\infty(\Omega)$
		&
		& $L^{n\over m}(\Omega)$
		& $1$
		\\
	\end{tabular}
\caption{Application of Theorem~A for the operator $H^{1}_{m/n}$ and John domain}
\label{tab:1}
\end{table*}

\subsection{Sobolev embeddings on Maz'ya classes}

Our next applications are in Sobolev embeddings on wider family of subsets so-called Maz'ya classes.

Let $\Omega$ be a domain in $\R^n$, $n\ge 2$, with a normalized Lebesgue measure, i.e.\ $|\Omega|=1$.
Define the perimeter of a measurable set $E$ in $\Omega$
$$
	P(E,\Omega)=\mathcal{H}^{n-1}(\Omega\cap \partial^M E)
$$
where $\partial^M E$ denotes the essential boundary of $E$.
The isoperimetric function $I_\Omega\colon [0,1]\to [0,\infty]$ of $\Omega$ is then
given by
$$
	I_\Omega(s)=\inf\bigl\{ P(E,\Omega),\;E\subseteq \Omega,\, s\le |E|\le \textstyle {1\over 2}\bigr\},
		\quad s\in\bigl[0,{1\over 2} \bigr]
$$
and $I_\Omega(s)=I_\Omega(1-s)$ if $s\in({1\over 2},1]$.

Given $\alpha \in[{1\over n'} , 1]$,  we denote by $\mathcal{J}_\alpha$
the~\df{Maz'ya class} of all  Euclidean domains $\Omega $
 in $\R^n$ such that
$$
I_\Omega (s) \geq C s^\alpha \quad \text{for }s \in \bigl[0,\textstyle{1\over 2}\bigr]
$$
for some positive constant $C$.

The reduction theorem in the class $\cal J_\alpha$
\citep[Theorem 6.4]{CPS} takes the following form.

		Let $n\in\N$, $n\geq 2$,  $m\in\N$, and $\alpha \in [\frac 1{n'}, 1)$.
		Let $\| \cdot \|_{X(0,1)}$ and $\| \cdot \|_{Y(0,1)}$ be rearrangement-invariant function norms.
		Assume that there exists a constant $C$ such that
		\begin{equation}
			\bigl\| H_{m(1-\alpha)}^1 f \bigr\|_{Y(0,1)}
				\leq C \| f \|_{X(0,1)}
					\label{E:eucl_condition-div}
		\end{equation}
		for every nonnegative $f\in X(0,1)$.  Then the Sobolev embedding
		\begin{equation}
			W^mX(\Omega ) \embed Y(\Omega )
				\label{E:eucl_embedd-div}
		\end{equation}
		holds  for every
		$\Omega \in \mathcal{J}_\alpha$.

		Conversely, if the Sobolev embedding \eqref{E:eucl_embedd-div} holds for every
		$\Omega \in\mathcal{J}_\alpha$, then the inequality \eqref{E:eucl_condition-div} holds.
	
Notice the main difference between this statement and reduction principle for John domains.
In the case of John domains the equivalence of
Sobolev embedding and boundedness of Hardy type operator holds for every
single domain~$\Omega$, while in the Maz'ya classes $\Omega$ has to
range among all domains in $\mathcal{I}_\alpha$.

Let us mention similar examples for Orlicz spaces.
Let $m$ be an integer and $\alpha\in[{1\over n'},1)$ such that $m(1-\alpha)<1$
and assume $p>1$ and $q\in\R$ or $p=1$ and $q\ge 0$.
By \citep[Theorem~6.12 and Example~6.14]{CPS}, we have

$$
W^mL^p{\log}^q L(\Omega ) \embed
\begin{cases}
	L^{\frac p{1-mp(1-\alpha )}}{\log}^\frac q{1-mp(1-\alpha)}L(\Omega ),
		& 1\le p<{1\over m(1-\alpha )},
		\\
	\exp L^{\frac 1{1-(1+q )m(1-\alpha)}} (\Omega ),
		& p={1\over m(1-\alpha)}, q < p-1,
		\\
	\exp\exp L^{\frac 1{1-m(1-\alpha)}} (\Omega ),
		& p={1\over m(1-\alpha )}, q = p-1,
		\\
	L^\infty (\Omega ),
		& p>{1\over m(1-\alpha)} \text{ or } p={1\over m(1-\alpha )}, q > p-1,
		\\
	\end{cases}
$$
and
$$
	W^mL^p\log^q\log L(\Omega) \embed
	\begin{cases}
		L^{p\over 1-mp(1-\alpha)}\log^{q\over 1-mp(1-\alpha)}\log L(\Omega),
			& 1\le p < {1\over m(1-\alpha)}, \\
		\exp \bigl(L^{1\over 1-m(1-\alpha)}\log^{mq(1-\alpha)\over 1-m(1-\alpha)}L\bigr)(\Omega),
			& p={1\over m(1-\alpha)}, \\
		L^\infty(\Omega),
			& p>{1\over m(1-\alpha)}.
	\end{cases}
$$
Moreover, the target spaces are optimal among all
Orlicz spaces, as $\Omega$ ranges in $\mathcal{J}_\alpha$.

Now one can apply Theorem~A for the operator $H^1_{m(1-\alpha)}$
in a analogous way as in Example~\ref{ex:applications} to investigate the optimal
Orlicz domains.

As computation shows, in the case $1\le p\le 1/m(1-\alpha)$ the optimality is
attained as $\Omega$ ranges through~$\mathcal{I}_\alpha$. In the remaining examples
there exists some $\Omega$ in $\mathcal{I}_\alpha$ such that any Orlicz domain
space in appropriate Sobolev embedding can be improved (see Table~\ref{tab:2}).

\begin{table*}
\centering
\renewcommand{\arraystretch}{1.4}
\footnotesize
	\begin{tabular}{llll}
	$Y(\Omega)$
		& 
		& $L^B(\Omega)$
		& $G(t)$
		\\ \hline
	$L^{p\over 1-mp(1-\alpha)}\log^{q\over 1-mp(1-\alpha)}L(\Omega)$
		& $1\le p< {1\over m(1-\alpha)}$
		& $L^p\log^qL(\Omega)$
		& $t^{{1\over m(1-\alpha)-1}+{p\over p-1}}\log^{q\over 1-p}(t)$
		\\
	$L^{1p\over 1-mp(1-\alpha)}\log^{1q\over 1-mp(1-\alpha)}\log L(\Omega)$
		& $1\le p< {1\over m(1-\alpha)}$
		& $L^p\log^q\log L(\Omega)$
		& $t^{{1\over m(1-\alpha)-1}+{p\over p-1}}\log^{q\over 1-p}\log(t)$
		\\
	$\exp L^{1\over 1-(1+q)m(1-\alpha)}(\Omega)$
		& $p = {1\over m(1-\alpha)}$, $q<p-1$
		&	$L^p\log^{1+q-p}L(\Omega)$
		& $\log^{1-(1+q)m(1-\alpha)\over 1-m(1-\alpha)}(t)$
		\\
	$\exp\exp L^{1\over 1-m(1-\alpha)}(\Omega)$
		& $p = {1\over m(1-\alpha)}$, $q=p-1$
		&	$L^p\log^{-q}\log L(\Omega)$
		& $\log\log (t)$
		\\
	$\exp \bigl(L^{1\over 1-m(1-\alpha)}\log^{mq(1-\alpha)\over 1-m(1-\alpha)}L\bigr)(\Omega)$
		& $p = {1\over m(1-\alpha)}$
		&	$L^p\log^{1-p} L\log^{q}\log L(\Omega)$
		& $\log(t)\log^{mq(1-\alpha)\over m(1-\alpha)-1}\log(t)$
		\\
	$L^\infty(\Omega)$
		&
		& $L^{m(1-\alpha)}(\Omega)$
		& $1$
		\end{tabular}
\caption{Application of Theorem~A for the operator $H^{1}_{m(1-\alpha)}$ and Maz'ya class}
\label{tab:2}
\end{table*}

\subsection{Sobolev trace embeddings}

Our last application concerns the Sobolev trace embeddings.

An open set $\Omega$ in $\R^n$ is said to have the cone property if there
exists a finite cone $\Lambda$ such that each point in $\Omega$ is the vertex
of a finite cone contained in $\Omega$ and congruent to $\Lambda$.

Given an integer $d$ such that $1\le d\le n$ we denote by $\Omega_d$ the nonempty
intersection of $\Omega$ with a $d$-dimensional affine subspace of $\R^n$.

The reduction principle for trace embeddings	\citep[Theorem 1.3]{CP2} now has
the following form.

		Let $\Omega$ be a bounded open set with cone property in $\R^n$, $n\ge 2$. Assume
		that $m\in\N$ and $d\in\N$ are such that $1\le d\le n$ and $d\ge n-m$.
		Let $\| \cdot \|_{X(0,1)}$ and $\| \cdot \|_{Y(0,1)}$ be rearrangement-invariant function norms.
		Then the following facts are equivalent.
		\begin{enumerate}[(i)]
		\item The inequality
		$$
			\bigl\|H^{n/d}_{m/n} f\bigr\|_{Y(0,1)} \le C \|f\|_{X(0,1)}
		$$
		holds for some constant $C$ and for every nonnegative $f\in X(0,1)$.
		\item The Sobolev trace embedding
		$$
			\Tr\colon W^mX(\Omega) \to Y(\Omega_d)
		$$
		holds.
		\end{enumerate}

Let $\Omega$ be a domain in $\R^n$ with cone property, $m\in\N$, $m<n$, and consider again the spaces
$L^p\log^qL(\Omega)$ or $L^p\log^q\log L(\Omega)$,
$p>1$ and $q\in\R$ or $p=1$ and $q\ge 0$.
By \citep[Theorem~5.2, Example~5.3 and Example~5.4]{CP2}
we have

$$
	\Tr\colon W^mL^p\log^qL(\Omega) \to
	\begin{cases}
		L^{pd\over n-mp}\log^{qd\over n-mp}L(\Omega_d),
			& 1\le p < {n\over m}, \\
		\exp L^{n\over n-m(1+q)}(\Omega_d),
			& p={n\over m}, q<{n\over m} -1, \\
		\exp\exp L^{n\over n-m}(\Omega_d),
			& p={n\over m}, q={n\over m} -1, \\
		L^\infty(\Omega_d),
			& p>{n\over m} \text{ or } p={n\over m}, q>{n\over m}-1,
	\end{cases}
$$
and
$$
	\Tr\colon W^mL^p\log^q\log L(\Omega) \to
	\begin{cases}
		L^{pd\over n-mp}\log^{qd\over n-mp}\log L(\Omega_d),
			& 1\le p < {n\over m}, \\
		\exp \bigl(L^{n\over n-m}\log^{mq\over n-m}L\bigr)(\Omega_d),
			& p={n\over m}, \\
		L^\infty(\Omega_d),
			& p>{n\over m},
	\end{cases}
$$
and the range spaces being optimal in the class of Orlicz spaces.

Now, using Theorem~A for the operator $H^{n/d}_{m/n}$, one can investigate the
optimal Orlicz domains. The situation is almost the same as in case of Sobolev
embedding and hence we just present the results (see Table~\ref{tab:3}).
Naturally, the optimality is attained only in the subcritical cases.

\begin{table*}
\centering
\renewcommand{\arraystretch}{1.4}
\begin{tabular}{llll}
$Y(\Omega)$
	& 
	& $L^B(\Omega)$
	& $G(t)$
	\\
	\hline
$L^{pd\over n-mp}\log^{qd\over n-mp}L(\Omega)$
	& $1\le p< {n\over m}$
	& $L^p\log^qL(\Omega)$
	& $t^{{n\over m-n}+{p\over p-1}}\log^{q\over 1-p}(t)$
	\\
$L^{pd\over n-mp}\log^{qd\over n-mp}\log L(\Omega)$
	& $1\le p< {n\over m}$
	& $L^p\log^q\log L(\Omega)$
	& $t^{{n\over m-n}+{p\over p-1}}\log^{q\over 1-p}\log(t)$
	\\
$\exp L^{n\over n-m(1+q)}(\Omega)$
	& $p = {n\over m}$, $q<{n\over m}-1$
	&	$L^p\log^{1+q-p}L(\Omega)$
	& $\log^{n-m(1+q)\over n-m}(t)$
	\\
$\exp\exp L^{n\over n-m}(\Omega)$
	& $p = {n\over m}$, $q={n\over m}-1$
	&	$L^p\log^{-q}\log L(\Omega)$
	& $\log\log (t)$
	\\
$\exp \bigl(L^{n\over n-m}\log^{mq\over n-m}L\bigr)(\Omega)$
	& $p = {n\over m}$
	&	$L^p\log^{1-p} L\log^{q}\log L(\Omega)$
	& $\log(t)\log^{mq\over m-n}\log(t)$
	\\
$L^\infty(\Omega)$
	&
	& $L^{n\over m}(\Omega)$
	& $1$
\end{tabular}
\caption{Application of Theorem~A for the operator $H^{n/d}_{m/n}$ and domain
with cone property} \label{tab:3}
\end{table*}

\subsection{Extension to other r.i.\ target spaces}

As we have seen in Example~\ref{ex:applications} a) in the case when the optimality
is attained one can extend the positive result to other r.i.\ target spaces.
Let us now look closer on this phenomenon.

Let $\alpha$ and $\beta$ be fixed and let $L^A(0,1)$ be an optimal Orlicz space
rendering the relation
$$
	H_\alpha^\beta\colon L^A(0,1) \to M(0,1)
$$
true, where $M(0,1)$ is a given Marcinkiewicz endpoint space.
We know from Theorem~A that not every Orlicz space is an optimal domain space;
such spaces are exactly those for which the supremum operator $S_\alpha$
is bounded on their associate space.

However, we can go the opposite direction. Suppose that $L^A(0,1)$ is a given
Orlicz space such that the operator $S_\alpha$ is bounded on $L^{\widetilde{A}}(0,1)$.
Now thanks to the result of \cite{KP2}, the operator $S_\alpha$
is bounded on some r.i.\ space $X'(0,1)$ if and only if
the $X(0,1)$ is optimal r.i.\ domain space for some r.i.\ target space.
By Proposition~\ref{prop:opt_ri_range}, the norm of the best r.i.\ target
space, say $Y_{L^A}(0,1)$,
is given by
$$
	\|f\|_{(Y_{L^{A}})'(0,1)} = \biggl\| t^{\alpha -1} \int_0^{t^{1\over\beta}} f^*(s)\,\d s \biggr\|_{L^{\widetilde{A}}(0,1)}.
$$
The fundamental function of $Y_{L^A}$, say $\varphi$, then satisfies (cf. \eqref{eq:chareq})
\begin{align*}
	\varphi(t)
		&\simeq t^{\beta(1-\alpha)}\,E_{\alpha -1}^{-1}(t^{-\beta})
			\\
		&\simeq t^{\beta(1-\alpha)}\,\widetilde{A}^{-1}(Kt^{-\beta}).
\end{align*}
Moreover, to the given Orlicz space $L^A(0,1)$, we are able to compute the appropriate Marcinkiewicz space $M(0,1)$.
If we take a look at the proof of Theorem~B again,
we observe that in the case of optimality, the inequality \eqref{eq:phiMvsE} becomes
actually equivalence, therefore the fundamental function of $M(0,1)$ is equivalent
to $\varphi$.

Consequently, we obtain that the space $L^A(0,1)$ is the optimal Orlicz domain
for every r.i.\ space $Y(0,1)$ satisfying
$$
	Y_{L^A}(0,1) \subseteq Y(0,1) \subseteq M(0,1).
$$

\begin{example} \label{ex:extension_to_Y}
Let $\Omega$ be a bounded Lipschitz domain in $\R^n$, $n\ge 2$, and $1<p<n$.
One can easily observe that $S_{1/n}$ is bounded on $L^{p'}(0,1)$, where
$p'=p/(p-1)$.
Then the optimal r.i.\ range space for the operator $H_{1/n}^1$ is the
Lorentz space $L^{p\ast,\,p}(0,1)$, where $p^* = np/(n-p)$. Its fundamental
function is equivalent to the power function $t^{1/p\ast}$ and therefore,
for every fixed $q\in [p,\infty]$,
the Lebesgue space $L^p(\Omega)$ is the largest Orlicz space which renders
the embedding
$$
	W^1L^p(\Omega) \embed L^{p\ast,\,q}(\Omega)
$$
true.

Similarly, for a given integer $1<m<n$, $1<p<n/m$
and $q\in[p,\infty]$
we obtain that $L^p(\Omega)$ is the largest Orlicz space in
$$
	W^m L^p(\Omega) \embed L^{{np\over n-mp},q}(\Omega).
$$
\end{example}

\section*{Acknowledgement}

I would like to express my thanks to Lubo\v s Pick for careful reading of this paper and many
valuable comments and suggestions. I also thank the referee for her/his useful comments.

\section*{References}

\bibliography{arxiv}

\begin{thebibliography}{10}
\expandafter\ifx\csname url\endcsname\relax
  \def\url#1{\texttt{#1}}\fi
\expandafter\ifx\csname urlprefix\endcsname\relax\def\urlprefix{URL }\fi
\expandafter\ifx\csname href\endcsname\relax
  \def\href#1#2{#2} \def\path#1{#1}\fi

\bibitem{CPS}
A.~Cianchi, L.~Pick, L.~Slav{\'{\i}}kov{\'a}, Higher-order {S}obolev embeddings
  and isoperimetric inequalities, Adv. Math. 273 (2015) 568--650.

\bibitem{EKP}
D.~E. Edmunds, R.~Kerman, L.~Pick, Optimal {S}obolev imbeddings involving
  rearrangement-invariant quasinorms, J. Funct. Anal. 170~(2) (2000) 307--355.

\bibitem{KP}
R.~Kerman, L.~Pick, Optimal {S}obolev imbeddings, Forum Math. 18~(4) (2006)
  535--570.

\bibitem{KP2}
R.~Kerman, L.~Pick, Optimal {S}obolev imbedding spaces, Studia Math. 192~(3)
  (2009) 195--217.

\bibitem{AC1}
A.~Cianchi, Orlicz-{S}obolev boundary trace embeddings, Math. Z. 266~(2) (2010)
  431--449.

\bibitem{AC2}
A.~Cianchi, Higher-order {S}obolev and {P}oincar\'e inequalities in {O}rlicz
  spaces, Forum Math. 18~(5) (2006) 745--767.

\bibitem{AC3}
A.~Cianchi, Optimal {O}rlicz-{S}obolev embeddings, Rev. Mat. Iberoamericana
  20~(2) (2004) 427--474.

\bibitem{AC}
A.~Cianchi, A sharp embedding theorem for {O}rlicz-{S}obolev spaces, Indiana
  Univ. Math. J. 45~(1) (1996) 39--65.

\bibitem{CP3}
A.~Cianchi, L.~Pick, Optimal {G}aussian {S}obolev embeddings, J. Funct. Anal.
  256~(11) (2009) 3588--3642.

\bibitem{G}
L.~Gross, Logarithmic {S}obolev inequalities, Amer. J. Math. 97~(4) (1975)
  1061--1083.

\bibitem{HMT}
J.~A. Hempel, G.~R. Morris, N.~S. Trudinger, On the sharpness of a limiting
  case of the {S}obolev imbedding theorem., Bull. Austral. Math. Soc. 3 (1970)
  369--373.

\bibitem{Po}
S.~I. Pokhozhaev, On eigenfunctions of the equation $\delta u+\lambda f(u)=0$,
  Dokl. Akad. Nauk SSSR 165 (1965) 36--39.

\bibitem{T}
N.~S. Trudinger, On imbeddings into {O}rlicz spaces and some applications, J.
  Math. Mech. 17 (1967) 473--483.

\bibitem{Yu}
V.~I. Yudovich, Some estimates connected with integral operators and with
  solutions of elliptic equations, Soviet Math. Doklady 2 (1961) 746--749.

\bibitem{Bonn}
L.~Pick, Optimal Sobolev Embeddings, Vol.~43 of Rudolph Lipschitz
  Vorlesungsreihe, Rheinische Friedrich-Wilhelms-Universit\"at, 2011.

\bibitem{Hu}
R.~A. Hunt, On {$L(p,\,q)$} spaces, Enseignement Math. (2) 12 (1966) 249--276.

\bibitem{Ma}
V.~G. Maz'ja, Sobolev spaces, Springer Series in Soviet Mathematics,
  Springer-Verlag, Berlin, 1985, translated from the Russian by T. O.
  Shaposhnikova.

\bibitem{ON}
R.~O'Neil, Convolution operators and {$L(p,\,q)$} spaces, Duke Math. J. 30
  (1963) 129--142.

\bibitem{Pe}
J.~Peetre, Espaces d'interpolation et th\'eor\`eme de {S}oboleff, Ann. Inst.
  Fourier (Grenoble) 16~(fasc. 1) (1966) 279--317.

\bibitem{CP}
A.~Cianchi, L.~Pick, Sobolev embeddings into {BMO}, {VMO}, and {$L_\infty$},
  Ark. Mat. 36~(2) (1998) 317--340.

\bibitem{CP2}
A.~Cianchi, L.~Pick, Optimal {S}obolev trace embeddings, Trans. Amer. Math.
  Soc. 368~(12) (2016) 8349--8382.

\bibitem{BS}
C.~Bennett, R.~Sharpley, Interpolation of operators, Vol. 129 of Pure and
  Applied Mathematics, Academic Press, Inc., Boston, MA, 1988.

\bibitem{FS}
L.~Pick, A.~Kufner, O.~John, S.~Fu{\v{c}}{\'{\i}}k, Function spaces. {V}ol. 1,
  extended Edition, Vol.~14 of De Gruyter Series in Nonlinear Analysis and
  Applications, Walter de Gruyter \& Co., Berlin, 2013.

\bibitem{L}
G.~G. Lorentz, On the theory of spaces {$\Lambda$}, Pacific J. Math. 1 (1951)
  411--429.

\bibitem{C}
A.-P. Calder{\'o}n, Spaces between {$L^{1}$} and {$L^{\infty }$} and the
  theorem of {M}arcinkiewicz, Studia Math. 26 (1966) 273--299.

\bibitem{KPP}
R.~Kerman, C.~Phipps, L.~Pick, Marcinkiewicz interpolation theorems for
  {O}rlicz and {L}orentz gamma spaces, Publ. Mat. 58~(1) (2014) 3--30.

\end{thebibliography}

\end{document}